\documentclass[10pt]{article}
\usepackage{amsfonts,epsfig}
\usepackage{amsfonts,epsfig,a4wide}
\usepackage{amsfonts,a4wide}
\usepackage{graphics}
\usepackage{booktabs}

\usepackage{amsmath,amssymb,amsthm}
\usepackage{url, cleveref}

\usepackage{float}

\usepackage{appendix}

\newtheorem{Th}{Theorem}
\newtheorem{Prop}{Proposition}
\newtheorem{Lemma}{Lemma}
\newtheorem{Cor}{Corollary}

\newtheorem*{Prob}{Problem}

\newtheorem{Ex}{Example}
\newtheorem{Rem}{Remark}
\newtheorem{Def}{Definition}

\newcommand{\veps}{\varepsilon}

\newcommand{\R}{\mathbb{R}}
\newcommand{\Z}{\mathbb{Z}}
\newcommand{\N}{\mathbb{N}}

\newcommand{\gst}{g^{\mathrm{std}}}

\newcommand{\gast}{\gamma^{\mathrm{std}}}

\newcommand{\dgast}{\dot{\gamma}^{\mathrm{std}}}

 \newcommand{\bbR}{\mathbb{R}}
\newcommand{\weg}[1]{}

\begin{document}


\title{Sophus Lie’s problem on two-dimensional metrics with projective symmetries: completing the local classification
}

\author{Gianni Manno\thanks{Dipartimento di Scienze Matematiche, Politecnico di Torino, Corso Duca degli Abruzzi 24, 10129 Torino (Italy), giovanni.manno@polito.it}, Vladimir S. Matveev\thanks{Institute of Mathematics, FSU Jena, 07737 Jena Germany,  vladimir.matveev@uni-jena.de}, Filippo Salis\thanks{Dipartimento di Scienze Matematiche, Politecnico di Torino, Corso Duca degli Abruzzi 24, 10129 Torino (Italy), filippo.salis@polito.it}}

\maketitle

\begin{abstract}

We complete the local classification, up to isometry, of $2$-dimensional pseudo-Riemannian metrics (i.e. both Riemannian and Lorentzian), admitting a projective symmetry algebra of dimension at least two. The new contribution is the treatment of the non-regular case: we obtain a complete list of mutually non-isometric normal forms in neighbourhoods of points where the action of the projective Lie symmetry algebra fails to be regular, including all singular behaviours that can occur.
Together with the known classification in the regular case, this completes the solution of the problem posed by Sophus Lie in 1882.
\end{abstract}

%

\section{Introduction}\label{sec.introduction}

\subsection{Basic definitions and conventions}

In this paper, the words ``\emph{metric}'' and ``\emph{pseudo-Riemannian metric}'' are used interchangeably to indicate both Riemannian and Lorentzian $2$-dimensional $C^\infty$-smooth metrics, unless otherwise specified. The word ``\emph{connection}'' stands for \emph{symmetric and affine connection.}

\smallskip\noindent
Two connections  $\Gamma$ and $\bar \Gamma$ on the same manifold $M$ are \emph{projectively equivalent} if they share the same unparametrised geodesics. The equivalence class of a connection w.r.t. projective equivalence will be called the \emph{projective class} of the connection. 
Similar definitions apply to metrics by considering their Levi-Civita connection.
A vector field on $M$ is called \emph{projective} with respect to a connection $\Gamma$ (resp. to a metric $g$) if its
local flow is made of geodesic-preserving transformations, i.e., transformations sending unparametrised geodesics of $\Gamma$ (resp. of $g$) into
unparametrised geodesics of $\Gamma$  (resp. of $g$). The bracket of two projective vector fields is itself a projective vector field, so that the set of projective vector fields has the structure of a Lie algebra, that turns out to be finite-dimensional. Of course, projectively equivalent connections (resp. metrics) share the same projective Lie algebra.
We shall consider mainly the Lie algebra of projective vector fields of a metric $g$, that we denote by $\mathfrak{p}= \mathfrak{p}(g)$. Likewise, a Lie subalgebra of $\mathfrak{p}(g)$ will be denoted by $\mathfrak{h}=\mathfrak{h}(g)$.
In this paper we study $2$-dimensional metrics $g$ with  $\dim\mathfrak{p}(g)\geq 2$.

\bigskip\noindent
\textbf{Conventions:} In what follows, all manifolds are assumed connected and, by a neighbourhood, we always mean an open neighbourhood.
By \emph{trivial} Killing vector field we mean the identically zero Killing vector field. We use the Einstein convention, i.e., repeated indices are implicitly summed over.


\subsection{Description of the main results}\label{sec.main.def.res}


The following  problem was  posed by Sophus Lie\footnote{German original from \cite{Lie}, Abschn. I, Nr. 4, Problem II:
\emph{Man soll die Form des Bogenelementes einer jeden Fl\"ache bestimmen,
deren geod\"atische Kurven mehrere infinitesimale Transformationen gestatten}.} in 1882:
\begin{Prob}[Lie]
Determine the local form of all $2$-dimensional metrics $g$ such that $\dim\mathfrak{p}(g)\geq 2$.
\end{Prob}
The reader can consult \cite{BMM,duna,MV,Matveev_first} for the history of the problem and for the description of the circle of ideas. For the connection with the results of Aminova one can see \cite{Aminova0,Aminova2}.  The question has been also studied in the context of parabolic geometries; see, e.g., \cite{melnik}, under the specific assumption that the symmetries (which, in the present setting, are projective vector fields) have a higher-order zero unless the geometry is flat.

\subsubsection{Starting point: solution of Lie's problem near almost every point}

\smallskip
The local description of $2$-dimensional metrics $g$ with  $\dim\mathfrak{p}(g)\geq 2$ in a neighbourhood of almost every point was obtained in \cite{BMM} and it is given in the following theorem.

\begin{Th}[\cite{BMM}] \label{th.Bryant.Manno.Matveev}
Let $(M,g)$ be a $2$-dimensional pseudo-Riemannian manifold with $\dim\mathfrak{p}(g)\geq 2$. Then, almost every point  of $M$ has a neighbourhood $U$ such that $g|_U$ has constant curvature, or such that there exists a coordinate system $(x,y)$ on $U$
where 
$g|_U$ takes one of the following forms:
\begin{enumerate}
\item Metrics with $\dim\mathfrak{p}(g)=2$.

\begin{enumerate}

\item $\veps_1 e^{(b+2)\,x}\,dx^2
    +\veps_2 e^{b\,x}\,dy^2 $, \label{1a}
where $b\in \mathbb{R}\setminus \{-2,0,1\}$
and $\veps_i\in \{-1,1\}$ are constants, \label{case1a}

\item \label{case1b}
 $a\left( \frac{ e^{(b+2) \, x} dx^2}
                           {(e^{b\,x } +\veps_2)^2}
     + \veps_1\frac{ e^{b\,x}dy^2}{e^{b\,x }+\veps_2}\right)$,
where $a\in\mathbb{R}\setminus\{0\}$, $b\in \mathbb{R}\setminus \{-2,0,1\}$,
and $\veps_i\in \{-1,1\}$ are constants, and

\item \label{case1c}
$a\left(\frac{e^{2\,x}dx^2}{x^2}+\veps \frac{dy^2}{x}\right)$,
where $a\in\mathbb{R}\setminus\{0\}$, and $\veps\in \{-1,1\}$ are constants.

\end{enumerate}

\item Metrics with $\dim\mathfrak{p}(g)=3$.

\begin{enumerate}

\item \label{case2a}
$\veps_1 e^{3x}dx^2+\veps_2 e^{x} dy^2$,
where $\veps_i\in \{-1,1\}$ are constants,

\item \label{case2b}
$a\left(\frac{e^{3x}dx^2}{(e^{x}+\veps_{2})^{2}} +
\veps_1\frac{e^{x}dy^2}{(e^{x}+\veps_{2})}\right)$,
where $a\in \mathbb{R}\setminus\{0\}$, $\veps_i\in \{-1,1\}$ are constants, and

\item \label{case2c}
$a\left(\frac{dx^2}{(2x^2+cx+\veps_2)^2 x}
          + \veps_1\frac{x dy^2}{ (2x^2+cx+\veps_2)}\right)$,
where $a>0$, $\veps_i\in \{-1,1\}$, $c\in \bbR$ are constants.

\end{enumerate}

\end{enumerate}
No two metrics from different entries of this list, or with different parameters, are locally isometric.
\end{Th}
%
%
%
\noindent

\begin{center}
{\underline{\textbf{Notation}:}} 
\begin{itemize}
\item We refer to $(M,g)$ as a $2$-dimensional pseudo-Riemannian manifold and
%
\item
$(\mathcal{M},\gst)$ as a \emph{standard model}, where $\gst$ is one of the metrics \ref{case1a}-\ref{case2c} of Theorem \ref{th.Bryant.Manno.Matveev} and $\mathcal{M}\subseteq\R^2$ the domain of $\gst$. Standard models of the same type but with different parameters are considered distinct.
\end{itemize}
\end{center}

\smallskip
\begin{Rem}\label{rem.proj.v.f.standard}
Note that all metrics $\gst$, i.e., all metrics \ref{case1a}-\ref{case2c} of Theorem \ref{th.Bryant.Manno.Matveev}, admit the Killing vector field $\partial_y$, which is unique up to a multiplicative constant. Moreover, this Killing vector field has length different from zero at every point.
A basis for the Lie algebra of projective vector fields of metrics \ref{case1a}, \ref{case1b}, \ref{case1c} is formed by
$$
\partial_y\,,\quad \partial_x+y\partial_y
$$
whereas for metrics \ref{case2a} and \ref{case2b} it is given by
$$
\partial_y\,,\quad \partial_x+y\partial_y\,,\quad 2y\partial_x+y^2\partial_y\,.
$$
A basis for the Lie algebra of projective vector fields of metrics  \ref{case2c} is formed by
\begin{itemize}
\item $\partial_y$, \,$-x\cos(y)\partial_x + \sin(y)\partial_y$, \, $x\sin(y)\partial_x +\cos(y)\partial_y$ \quad if $\varepsilon_1\varepsilon_2=1$;
\item $\partial_y$, \, $-xe^y\partial_x+e^y\partial_y$, \,$xe^{-y}\partial_x+e^{-y}\partial_y$  \quad if $\varepsilon_1\varepsilon_2=-1$.
\end{itemize}
We note that, in all the cases, the Lie algebra of projective vector fields acts locally transitively on $\mathcal{M}$ as it possesses only $2$-dimensional orbits. If the Lie algebra $\mathfrak{p}(g)$ of projective vector fields of a $2$-dimensional metric $g$ has only $1$-dimensional orbits, then the metric $g$ has constant curvature (see \cite[Lemma 1]{BMM}).

\noindent
Furthermore, any projective vector field of $\gst|_V$, where $V$ is an open subset of $\mathcal{M}$, is the restriction to $V$ of a projective vector field of $(\mathcal{M},\gst)$, i.e., $\mathfrak{p}(\gst|_V)=\big(\mathfrak{p}(\gst)\big)|_V$. This last property will be justified in Section \ref{sec.proj.conn}, where the method for computing the above projective vector fields is briefly illustrated.
\end{Rem}

\subsubsection{Description of the new results: completion of the solution of Lie’s problem}

The aim of the present paper is to complete the list of Theorem \ref{th.Bryant.Manno.Matveev} by finding normal forms of metrics $g$ on a $2$-dimensional manifold $M$ with $\dim\mathfrak{p}(g)\geq 2$ near the points that were not considered in that theorem. To this purpose, we say that a point $p\in (M,g)$ is \emph{regular} if there exists a neighbourhood $U$ of $p$ such that $g|_{U}$ is described by a metric of Theorem \ref{th.Bryant.Manno.Matveev}, i.e., $g|_{U}$ is either of constant curvature or 
%
there exists a diffeomorphism $\varphi$ from $U$ to a neighbourhood $V$ of a standard model $(\mathcal{M},\gst)$ such that $\varphi^*(\gst|_V)=g|_U$. In the latter situation, we can refer to such points as \emph{points of type \ref{case1a},\dots,\ref{case2c}} \emph{or of constant curvature}, depending on the case.
Points which are not regular will be called \emph{singular}. 

\smallskip
On account of the above definitions, we introduce the following sets:
$$
\emph{\text{Regular locus:}}=\{p\in M\,\,|\,\, \text{$p$ is a regular point}\}\,, \quad
\emph{\text{Singular locus:}}=\{p\in M\,\,|\,\, \text{$p$ is a singular point}\}\,.
$$

\smallskip
In view of Remark \ref{rem.proj.v.f.standard}, points $p\in M$ such that the dimension of the vector distribution on $M$ associated to $\mathfrak{p}(g)$ drops at $p$ are candidates to be singular, and actually they are, but we will see, cf. Section \ref{sec.proj.v.f.standard.2} below, that they are not the only ones.

\medskip\noindent
Now we are ready to state the main result of the paper, that is contained in the following theorem.

\begin{Th}\label{th.main}
%
Let $(M,g)$ be a $2$-dimensional pseudo-Riemannian manifold with $\dim\mathfrak p(g)\geq 2$. Then, every singular point of $M$ has a neighbourhood $U$ 
endowed 
with a coordinate system $(x,y)$ where
$g|_U$ takes one of the following forms:
\begin{enumerate}
\item Metrics with $\dim\mathfrak{p}(g)=2$.

\begin{enumerate}

\item\label{case1a.main} $\varepsilon^{h+1}(1+yx^{h+1})^{-\frac{h+2}{h+1}}dxdy$,

\item\label{case1b1.main}
$\kappa \left(\varepsilon_1\frac{d{x}^2}{(1+{x}^{h+2})^2}
+
\varepsilon_2\,
\frac{d{y}^2}{1+{x}^{h+2}}\right)$,

\item\label{case1b1bis.main}
$
\kappa\left(\varepsilon_1\frac{d{x}^2}{(1-{x}^{2h+2})^2}
+
\varepsilon_2\,
\frac{d{y}^2}{1-{x}^{2h+2}}\right)
$,

\item\label{case1b2.main} 
$
\kappa\left(\frac{2}{(1+ x^{h+2})^{\frac32}}\,d{x}\,d{y}
+
\veps\,
\frac{{x}^{h+2}}{1+ x^{h+2}}\,d{y}^{2}\right)
$,

\item\label{case1b2bis.main} 
$\kappa
\left(\frac{2}{(1- x^{2h+2})^{\frac32}}\,d{x}\,d{y}
+
\veps\,
\frac{{x}^{2h+2}}{1- x^{2h+2}}\,d{y}^{2}\right)
$,
\end{enumerate}
where $\veps,\veps_i\in\{-1,1\}$, $h\in\mathbb{N}$,  $\kappa>0$, with the additional convention that, in case \ref{case1a.main}, $\varepsilon=1$ when $h$ is odd.

\item Metrics with $\dim\mathfrak{p}(g)=3$. \label{case2c.main}


\begin{enumerate}

\item \label{case2c.singular}
$\kappa\left(\frac{\veps_1}{f(s)}\left(dx^2+dy^2 \right)
+\,
\frac{h(s)}{f^2(s)} (xdx+ydy)^2\right)$, 

\item \label{case2c.singular.bis}
$\kappa\left(\frac{1}{f(t)}\left(dx^2-dy^2 \right)
+\,
\frac{h(t)}{f^2(t)} (xdx-ydy)^2\right)$, 

\end{enumerate}
where
$$
f(w)=2\veps_2w^2+c\veps_2 w+1\,,\quad h(w)=-2\veps_2 w-c\veps_2 \,,\quad s=\veps_1(x^2+y^2)\,,\quad t=x^2-y^2
$$
and $\veps_i\in\{-1,1\}$, $\kappa>0$, $c\in\R$.
\end{enumerate}
No two metrics from different entries of this list, or with different parameters, are locally isometric.
\end{Th}
Theorem \ref{th.Bryant.Manno.Matveev} together with Theorem \ref{th.main} completely solves Lie's problem.

\medskip\noindent
A significant part of the proof of Theorem \ref{th.main} is that of showing that points of different types cannot coexist on the same manifold. More precisely, we have the following theorem.
\begin{Th}\label{th.main.3}
Let $(M,g)$ be a $2$-dimensional pseudo-Riemannian manifold with $\dim\mathfrak{p}(g)\geq 2$. Then regular points of $M$ are of the same type.
\end{Th}
Also, a direct computation shows that all metrics \ref{case1a.main}-\ref{case2c.singular.bis} of Theorem \ref{th.main} admit a Killing vector field, unique up to a multiplicative constant. Indeed, another important part of the proof of Theorem \ref{th.main} is devoted to establishing the following theorem, which will be useful, among other things, for characterizing the singular locus, see Section \ref{sec.proj.v.f.standard.2} below.
\begin{Th}\label{th.main.2}
Let $(M,g)$ be a $2$-dimensional pseudo-Riemannian manifold with $\dim\mathfrak{p}(g)\geq 2$. Then in a 
neighbourhood of every point it admits a nontrivial Killing vector field.\footnote{In a sufficiently small neighbourhood of every regular point,
Theorem \ref{th.main.2} follows  from Theorem
\ref{th.Bryant.Manno.Matveev}, cf. Remark \ref{rem.proj.v.f.standard}. 
}
\end{Th}
%
%
%

%
\subsubsection{Projective vector fields of metrics of Theorem \ref{th.main} and geometric characterization of the singular locus}\label{sec.proj.v.f.standard.2}

Using Theorems \ref{th.main} and \ref{th.main.2}, we can provide the following geometric characterization of singular points: a point $p$ of a $2$-dimensional pseudo-Riemannian manifold $(M,g)$ with $\dim\mathfrak{p}(g)\geq 2$ is singular if either the Lie algebra action of $\mathfrak{p}(g)$ on $M$ is singular, i.e.,  the dimension of the vector distribution on $M$ associated to $\mathfrak{p}(g)$ drops at $p$, or the Killing vector field (whose existence is guaranteed by Theorem \ref{th.main.2}) has length equal to zero at $p$.

\smallskip\noindent
For the sake of completeness, we list below the projective vector fields, including in particular the Killing vector fields, of the metrics \ref{case1a.main}-\ref{case2c.singular.bis} in Theorem \ref{th.main}. As mentioned above, the method used to obtain them is briefly described in Section \ref{sec.proj.conn}. This will allow us to provide, at the end of this section, a precise case-by-case description of the singular locus.


\begin{itemize}
\item a Killing vector field of metrics \ref{case1a.main} is $x^{h+2}\partial_x+(h+1)\partial_y$;
\item a Killing vector field of metrics \ref{case1b1.main}-\ref{case1b2bis.main}   is $\partial_y$;
\item a Killing vector field of metrics \ref{case2c.singular} is  $y\partial_x-x\partial_y$;
\item a Killing vector field of metrics \ref{case2c.singular.bis} is  $y\partial_x+x\partial_y$;
\end{itemize}

\medskip\noindent
Below we give additional projective vector fields, that, together with the Killing vector fields listed above, form a basis of $\mathfrak{p}(g)$ casewise.
\begin{itemize}
\item an additional projective field of metrics \ref{case1a.main} is $-\frac{x}{h+1}\partial_x+y\partial_y$, that turns out to be a homothety;
\item an additional projective field of metrics \ref{case1b1.main} is $-\frac{2x}{h}\partial_x+y\partial_y$;
\item an additional projective field of metrics \ref{case1b1bis.main} is $-\frac{x}{h}\partial_x+y\partial_y$;
\item an additional projective field of metrics \ref{case1b2.main} is $\frac{2\veps}{h+2}x\partial_x + \left(\displaystyle\int\frac{dx}{(1+x^{h+2})^{\frac32}} - 2\veps\frac{h+1}{h+2}y\right)\partial_y$;
\item an additional projective field of metrics \ref{case1b2bis.main} is $-\frac{\veps}{h+1}x\partial_x + \left(\displaystyle\int\frac{dx}{(1-x^{2h+2})^{\frac32}} +\veps\frac{2h+1}{h+1}y\right)\partial_y$;
\item additional projective fields of metrics \ref{case2c.singular}-\ref{case2c.singular.bis} are $y\partial_x$ and $x\partial_x-y\partial_y$;
%
%
\end{itemize}
We are now in a position to provide a precise description of the singular locus of metrics \ref{case1a.main}-\ref{case2c.singular.bis} of Theorem \ref{th.main}.

\begin{itemize}
\item[-] The singular locus of metrics \ref{case1a.main}-\ref{case1b2bis.main} is described by $x=0$, on which the projective vector fields become collinear: indeed, on $x=0$ they span a $1$-dimensional distribution whereas outside a $2$-dimensional one. 

\item[-]
The singular locus of metrics \ref{case2c.singular} is the origin: at this point all projective vector fields vanish, whereas on the complementary set they span a $2$-dimensional distribution. 

\item[-]
The singular locus of metrics \ref{case2c.singular.bis} is described by $(x-y)(x+y)=0$, on which the length of the Killing vector field vanishes. Also, all projective vector fields vanish at the origin, whereas outside they span a $2$-dimensional distribution.


\end{itemize}

\section{Proof of Theorems \ref{th.main}, \ref{th.main.3} and \ref{th.main.2}}\label{sec.proof}

\subsection{Ideas, techniques and scheme of the paper}\label{sec.outline}

Let $(M,g)$ be a pseudo-Riemannian manifold with $\dim\mathfrak{p}(g)\geq 2$. In \cite{BMM}, it was shown that the set of regular points of $(M,g)$ is dense. Thus, a natural way to construct examples of metrics admitting two or three projective vector fields near singular points is to glue neighbourhoods endowed with the metrics appearing in Theorem \ref{th.Bryant.Manno.Matveev} into a larger neighbourhood.

\noindent
Because of the uniqueness statement in Theorem \ref{th.Bryant.Manno.Matveev}, the neighbourhoods cannot overlap; rather, they are expected to be glued along a kind of boundary. In this setting, the natural candidate for such a boundary is the set of points at which the metrics cease to be defined, namely the points where the denominator in the formulas for the metrics  in Theorem \ref{th.Bryant.Manno.Matveev} vanishes, together with the points at infinity, that is, points for which either the coordinate $x$ or the coordinate $y$ tends to infinity. In the proof of Theorem \ref{th.main}, we essentially show that all metrics described there arise in this way. However, such an analysis is not straightforward, since, for instance, one has no a priori control over the geometry or topology of the singular locus, nor over the possible adjacencies between regions corresponding to different standard models. In this regard, it is worth stressing that most parameters of a standard model are continuous, yielding a priori infinitely many possible gluings between different models. Furthermore, another difficulty is that, a priori, it is not clear whether, in the case $\dim\mathfrak p(g)= 2$, there exists a neighbourhood $U\subset M$ such that $\dim\mathfrak p(g|_U)> 2$.

\smallskip\noindent
To gain a better and deeper understanding of the main ideas, strategies and difficulties underlying the proofs of Theorems \ref{th.main}, \ref{th.main.3} and \ref{th.main.2}, we need the following definitions, which play a key role in the present paper.
\begin{Def}\label{def:modelled}
Let $(M,g)$ be a pseudo-Riemannian manifold with $\dim\mathfrak p(g)\geq 2$.
A geodesic $\gamma:[0,1]\to M$ 
will be referred to as \emph{modelled geodesic} if there exists a neighbourhood $U$ of $M$ containing $\gamma(I)$, with $I=(0,1)$, and an isometry from $U$ into a neighbourhood of a standard model  $(\mathcal M,\gst)$.
\end{Def}
We will see (cf. Proposition \ref{lem:no-interior-accumulation}) that the image of the interior of a modelled geodesic issuing from a singular point, under the isometry of Definition \ref{def:modelled}, is a geodesic of a standard model that must either converge to the boundary of the standard model or escape every compact subset within finite affine parameter. This justifies the following definition.
\begin{Def}\label{def:gast}
Let $(\mathcal M,\gst)$ be a standard model.
A geodesic $\gast:(0,1)\to \mathcal M$ 
will be referred to as \emph{suspicious geodesic} if, as $t\to 0^+$, either $\gast(t)$ tends to $\partial \mathcal M$ or $\gast(t)$ is unbounded.
\end{Def}
Thus, under the isometry of Definition \ref{def:modelled}, the interior part of every modelled geodesic issuing from a singular point gives rise to a suspicious one; however, not every suspicious geodesic is obtained in this way.

\smallskip
We observe that modelled geodesics issuing from regular points always exist:
a first problem is to determine under which conditions a modelled geodesic issuing from a given singular point exists. After some preliminary results, presented in Section \ref{sec.proj.conn}, we address this problem in Section \ref{sec:modelled.geod}, whose main outcomes are Propositions \ref{prop:admissible.arc} and \ref{prop:admissible.arc.2}. Proposition~\ref{prop:admissible.arc} shows that any geodesic starting at a singular point and remaining in the regular locus thereafter is modelled. Proposition~\ref{prop:admissible.arc.2}, in turn, ensures that any geodesic issuing from a singular point and meeting the regular locus contains a modelled geodesic arc. The situation described above is summarized in Figure \ref{fig.mod.geod} below.

\begin{figure}[H]
\begin{center}
\includegraphics[width=0.5\textwidth]{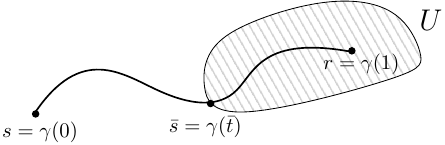}
\caption{This picture summarizes the content of Propositions \ref{prop:admissible.arc} and \ref{prop:admissible.arc.2}. A geodesic curve $\gamma:[0,1]\to M$, connecting a singular point $s=\gamma(0)$ to a regular point $r=\gamma(1)$, contains a geodesic arc $\gamma(t)$, $\bar t \leq t \leq 1$, which, after reparametrizing $[\bar t,1]$ onto $[0,1]$, is a modelled geodesic issuing from the singular point $\bar s=\gamma(\bar t)$ in view of Proposition  \ref{prop:admissible.arc}. In other words, there exists a neighbourhood $U\subset M$ of this arc, with $\bar s$ removed, that is isometric to a neighbourhood in a standard model. However, the distribution of singular points on the arc $(s,\bar{s})$ can be determined only by a further analysis carried out in Section \ref{sec.regular.singular}.}\label{fig.mod.geod}
\end{center}
\end{figure}

\vspace{-0.3cm}

The proofs of Proposition \ref{prop:admissible.arc} and  \ref{prop:admissible.arc.2} are divided into three steps. First, in Section \ref{sec.extension.iso}, we establish a propagation result for local isometries along the flows of projective vector fields. Second, in Section \ref{sec.obstruction}, we analyse the only obstruction to applying
this propagation argument along a geodesic, namely the possible degeneration of a $2$-dimensional projective subalgebra to a $1$-dimensional distribution. This leads us to introduce and study the \emph{non-transitivity locus} of  $2$-dimensional projective subalgebra, that allows us to control the aforementioned degeneration phenomenon. Finally, in Section \ref{sec.prop.3.4}, we combine these ingredients to prove Propositions \ref{prop:admissible.arc} and \ref{prop:admissible.arc.2}.


\smallskip
Let $\gamma$ be a modelled geodesic arc issuing from a singular point, for instance, the geodesic arc $[\bar{s},r]$ of Figure \ref{fig.mod.geod}.
After identifying a neighbourhood $U$ of $\gamma$ with a neighbourhood of a standard model $(\mathcal{M},\gst)$ via the isometry constructed in Proposition \ref{prop:admissible.arc.2}, the behaviour of the metric 
near the singular point $\gamma(0)$ can be understood by studying the
image of $\gamma$ (via this isometry) that, as we said, is a suspicious geodesic in view of Proposition \ref{lem:no-interior-accumulation}.  

\smallskip
In Section \ref{sec.behaviour}, based on the above discussions, we provide a tool to verify necessary conditions for a suspicious geodesic $\gast$ to come from a modelled one issuing from a singular point. More precisely, one has to check that metric invariants such as the scalar curvature $R$, its length $I$ and its laplacian $\Delta R$ be bounded along such geodesics.

\smallskip
In Section \ref{sec.extendability} the above reasoning is used to select the candidate standard metrics that can be glued in a larger domain containing singular points. Indeed, referring to Figure \ref{fig.mod.geod}, we see when the aforementioned metric invariants explode when approaching the singular point $\bar{s}$. In particular, as a first step, in Section \ref{sec.1c.2a.2b.not} we see that at least one of them is unbounded for metrics \ref{case1c}, \ref{case2a} and \ref{case2b} of Theorem \ref{th.Bryant.Manno.Matveev}. Therefore, if $(M,g)$ contains regular points of type  \ref{case1c}, \ref{case2a} or \ref{case2b}, then it cannot contain singular points.  Actually a stronger statement holds: once regular points of one of the types \ref{case1c}, \ref{case2a}, or \ref{case2b} occur, no regular points of any other type can occur on the same manifold, cf. Proposition \ref{impossible}. 
The remainder of Section \ref{sec.extendability} is devoted to the cases not ruled out by the obstruction described above. More precisely, in Sections \ref{sec:1a}, \ref{sec:1b}, and \ref{sec:2c}, we consider, respectively, the cases in which $(M,g)$ contains regular points of type \ref{case1a}, \ref{case1b}, and \ref{case2c}. In each case, the strategy is the same: we analyse which suspicious geodesics may represent modelled geodesics approaching a singular point. In particular, the boundedness of metric invariants, together with the requirement that the affine parameter remain finite, yields necessary restrictions on the parameters of the standard models which can occur near a singular point. 

In Section \ref{sec.regular.singular} we study more deeply the structure of the regular and singular locus of $(M,g)$.
Referring again to Figure \ref{fig.mod.geod}, one has to understand how singular points can be distributed along $\gamma$. For all admissible types, this analysis shows in particular that singular points are isolated along geodesics transverse to the singular locus. Concerning the regular locus, the main achievement is that, once regular points of one of the types \ref{case1a}, \ref{case1b}, or \ref{case2c} occur, no regular points of any other type can occur on the same manifold, see the end of Section \ref{sec:regular.locus.2c}. This analysis, together with the results contained in Proposition \ref{impossible}, enters the proof of Theorem \ref{th.main.3}.
Moreover, while the existence of a nontrivial Killing vector field near regular points follows from the standard models, near singular points it is obtained through this analysis together with the results contained in the Appendix, leading to the completion of the proof of Theorem \ref{th.main.2} at the end of Section \ref{sec:singular.locus.2c}. 
This analysis also shows that the singular locus is organized into the union of integral submanifolds of the Killing vector field, in particular, it can be a point, a curve or a union of two curves intersecting at a single point. Then, a natural problem is to see what happens 
when two regular regions meet along the singular locus: we study it by comparing smooth geometric quantities across such locus. This yields additional compatibility conditions on the parameters of the models on either side of the gluing.

In Section \ref{sec:final.th.main}, taking into account that the structure of the singular locus is understood, we derive the local normal forms of Theorem \ref{th.main}.
The preceding analysis provides the natural geometric data for this purpose: the Killing vector field, the modelled geodesic arcs issuing from the singular locus and the restrictions on the parameters of regular points. 
We are then naturally led to consider some coordinates reflecting this geometric structure.
In such coordinates, the regular pieces adjacent to the singular locus can be compared explicitly, and the compatibility conditions obtained above ensure that the metric extends smoothly across the singular locus. 

\subsection{Some preliminary known facts we will use in the proof}\label{sec.proj.conn}

It is well known, at least since Beltrami \cite{Beltrami},
that geodesics of a connection
$\Gamma=\Gamma_{ij}^k$ on a coordinate neighbourhood of $\R^2(x,y)$ and solutions of the ODE
\begin{equation}\label{eq.proj.conn.associated}
y_{xx}=\underbrace{-\Gamma^2_{11}}_{F^0} + \underbrace{(\Gamma^1_{11}-2\Gamma^2_{12})}_{F^1}y_x
+\underbrace{(2\Gamma^1_{12}-\Gamma^2_{22})}_{F^2}y_x^2 +\underbrace{ \Gamma^1_{22}}_{F^3} y_x^3,
\end{equation}
are closely related: namely, for every solution $y(x)$ of \eqref{eq.proj.conn.associated} the curve
$(x,y(x))$ is a reparametrised geodesic. Equation \eqref{eq.proj.conn.associated} is classically  called the
\emph{projective connection} associated with the connection $\Gamma$.
Conversely, in a local coordinate system,
any choice of  functions $F^0$, $F^1$, $F^2$ and $F^3$  determines a class of connections having the same (unparametrised) geodesics.
Every projective vector field of the connection $\Gamma$ is a \emph{point
symmetry} of the ODE \eqref{eq.proj.conn.associated}, and vice
versa. The vector field 
$$
X=X^1\partial_x+X^2\partial_y
$$ 
is a point symmetry of \eqref{eq.proj.conn.associated} if and only if its components satisfy the PDEs system
\begin{equation}\label{eq:sys.proj.fields}
\left.
\begin{array}{r}
X^2_{xx} -2F^0X^1_x -F^1X^2_x +F^0X^2_y -F^0_xX^1 -F^0_yX^2=0
\\[0.2cm]
X^1_{xx} -2X^2_{xy} +F^1X^1_x +3F^0X^1_y +2F^2X^2_x +F^1_xX^1 +F^1_yX^2=0
\\[0.2cm]
2X^1_{xy} -X^2_{yy} +2F^1X^1_y +3F^3X^2_x +F^2X^2_y +F^2_xX^1 +F^2_yX^2=0
\\[0.2cm]
X^1_{yy} -F^3X^1_x +F^2X^1_y +2F^3X^2_y +F^3_xX^1 +F^3_yX^2=0
\end{array}
\right\}
\end{equation}
that is of finite type with space of solutions  at most $8$-dimensional (see Section 2.2.1 of \cite{BMM}).

\smallskip\noindent
Focusing on the standard models $(\mathcal{M},\gst)$, direct computations based on formula \eqref{eq.proj.conn.associated} show that the following ODE 
\begin{equation}\label{eq.ref.proj.conn.2}
y_{xx} = \left(1-\frac12 b\right)y_x-\frac12 \varepsilon_1\varepsilon_2b e^{-2x}y_x^3
\end{equation}
is the projective connection associated to the Levi-Civita connection
\begin{itemize}
\item of both the metrics \ref{case1a} and \ref{case1b} for $b\notin\{-2,0,1\}$ of Theorem \ref{th.Bryant.Manno.Matveev};
\item of both the metrics \ref{case2a} and \ref{case2b} for $b=1$ of Theorem \ref{th.Bryant.Manno.Matveev};
\item of the metric \ref{case1c} of Theorem \ref{th.Bryant.Manno.Matveev} (putting in this case $\varepsilon_1\varepsilon_2=\varepsilon$) for $b=-2$.
\end{itemize}
Instead, the projective connection associated to the Levi-Civita connection of metrics \ref{case2c} of Theorem \ref{th.Bryant.Manno.Matveev} is independent of the constant $c$ and it is equal to
\begin{equation}\label{eq.ref.proj.conn.3}
y_{xx}= -\frac{3}{2x}y_x - \frac12\varepsilon_1 x(\varepsilon_2-2x^2)y_x^3\,.
\end{equation}
Point symmetries of \eqref{eq.ref.proj.conn.2} and \eqref{eq.ref.proj.conn.3} can be computed by an explicit integration of System \eqref{eq:sys.proj.fields}, thus obtaining
the projective vector fields listed in Remark \ref{rem.proj.v.f.standard}: we note that in this case any local solution of System \eqref{eq:sys.proj.fields} is a restriction of a solution defined on the whole of $\mathcal{M}$. 

\smallskip\noindent
By similar reasonings we arrive also at the list of projective vector fields given in Section \ref{sec.proj.v.f.standard.2}.


\medskip
Equation \eqref{eq.proj.conn.associated} implies that two connections $\Gamma$ and $\bar \Gamma$ on the same manifold are projectively equivalent if and only if  they correspond to the same  ODE \eqref{eq.proj.conn.associated}, i.e., if
$$
-\Gamma^2_{11} =  F^0 = -\bar \Gamma^2_{11}, \   \Gamma^1_{11}-2\Gamma^2_{12}= F^1 = \bar \Gamma^1_{11}-2\bar \Gamma^2_{12}, \  2\Gamma^1_{12}-\Gamma^2_{22}= {F^2} =2\bar \Gamma^1_{12}-\bar \Gamma^2_{22}, \
 \Gamma^1_{22}= F^3= \bar \Gamma^1_{22}.
$$
Thus, in  a local coordinate system, a class of projectively equivalent connections is essentially the same as ODE \eqref{eq.proj.conn.associated}, i.e., is uniquely determined by  4 functions $F^0,...,F^3$.

A classical question is to understand which connections in a given projective class
are Levi-Civita connections of some metric.
The reformulation of this question as the existence of a (non-degenerate) solution of a linear system of PDEs is again classical and is due to
Liouville \cite{liouville} (see also Lemma $5$ of \cite{BMM}): he has shown that the Levi-Civita connection of a metric $g$ lies in a given projective class if and only if the entries $a_{11}(x,y), a_{12}(x,y), a_{22}(x,y)$ of the symmetric
matrix 
\begin{equation}\label{a.from.g}
a= \begin{pmatrix}a_{11} &a_{12} \\ a_{12}  &a_{22}\end{pmatrix} = \det(g)^{-\frac23}
 g
\end{equation}
satisfy the following system:
\begin{equation}\label{lin1}
\left.
\renewcommand{\arraystretch}{1.5}
\begin{array}{rcc}
{a_{11}}_x-\tfrac{2}{3}\,F^1\,a_{11} +2\,F^0\,a_{12}&=&0\\
{a_{11}}_y+2\,{a_{12}}_x
-\tfrac{4}{3}\,F^2\,a_{11}+\tfrac{2}{3}\,F^1\,a_{12}+2\,F^0\,a_{22}&=&0\\
2\,{a_{12}}_y+{a_{22}}_x
-2\,F^3\,a_{11}-\tfrac{2}{3}\,F^2\,a_{12}+\tfrac{4}{3}\,F^1\,a_{22}&=&0\\
{a_{22}}_y-2\,F^3\,a_{12}+\tfrac{2}{3}\,F^2\,a_{22} &=&0
\end{array} \right\}\,.
\end{equation}
Note that the matrix \eqref{a.from.g} is non-degenerate, and that  the transformation
  $g\mapsto a=  \det(g)^{-\frac23}
 g$  is invertible: the inverse transformation is given by 
\begin{equation*}
a\mapsto g= \det(a)^{-2}a.
\end{equation*}  
Thus, non-degenerate solutions of \eqref{lin1} are in 1:1 correspondence with metrics whose Levi-Civita connection lies in the  projective class used to construct system \eqref{lin1}. 
Note that \eqref{a.from.g} can be viewed as a section of the tensor bundle
$$
S^2T^*M  \otimes \mathrm{Vol}^{-\frac43}\,,
$$
where $S^2T^*M$ is the bundle of symmetric $(0,2)$-tensors on $M$ and
 $\mathrm{Vol}^{-\frac43}$ denotes the $1$-dimensional bundle of volume forms $\Lambda^2T^*M$ of $M$ with weight $-\frac43$ (we assume that the manifold is oriented; in fact, since we are working locally, the problems related to the fact that there is no trivialization of the bundle of the volume forms do not appear). 
%
%
%
%
Furthermore, the Lie derivative $\mathcal{L}_Xa$ of a section $a=\det(g)^{-\frac23}g$ of $S^2T^*M  \otimes \mathrm{Vol}^{-\frac43}$ along a vector field $X$ is a well-defined section of $S^2T^*M  \otimes \mathrm{Vol}^{-\frac43}$:
\begin{equation*}
\mathcal{L}_Xa=\det(g)^{-\frac23}\left(\mathcal{L}_Xg-\frac23\mathrm{trace}_g(\mathcal{L}_Xg)\,g\right)\,.
\end{equation*}
In particular, if the vector field $X$ is projective and  $a=(a_{ij})$ is  a solution of \eqref{lin1}, then $\mathcal{L}_X a$  is also a solution of \eqref{lin1} as such system  depends on the projective connection only and it is therefore projectively invariant.

Let us also note that the system \eqref{lin1} is an overdetermined linear PDE-system of finite type  (it closes after two prolongations and its second prolongation can be viewed as a particular linear connection on a certain $6$-dimensional vector bundle, in the sense that parallel sections of this connection are in 1:1 correspondence with the solutions of \eqref{lin1}, see \cite{eastwood}) which implies some useful properties, contained in the next two propositions.
\begin{Prop}\label{prop:solutions}
The space of solutions of \eqref{lin1} is a  finite-dimensional (in the present case, at most,  6-dimensional)  vector space, in particular, any solution $a=(a_{ij}):\Omega \subseteq \R^2\to \R^3$ is completely determined by the second-order jet in an arbitrary point $p\in\Omega$. Hence,  if a solution of \eqref{lin1} vanishes on a neighbourhood, it vanishes everywhere.  Consequently, if two solutions of \eqref{lin1} coincide on a neighbourhood, they coincide everywhere.
\end{Prop}
%
%
 \begin{Prop}\label{homo}
 If two projectively equivalent metrics are proportional with a constant coefficient
 on an open set,  then
they are proportional with the same constant coefficient on the whole manifold. Moreover, if a projective vector field $X$ of $(M,g)$ is a homothety (resp. Killing) in a neighbourhood, then it is a homothety (resp. Killing) on $M$.
\end{Prop}
\begin{Cor}\label{cor.proj.field.vanishing}
If a projective vector field $X$ of $(M,g)$ vanishes on a neighbourhood $U\subseteq M$ then it vanishes everywhere.
\end{Cor}
%
\begin{Rem}
Actually the above proposition is true in any dimension greater than $1$, as a projectively invariant linear system with similar properties as system \eqref{lin1} holds in arbitrary dimension, see \cite{eastwood}.
\end{Rem}

\subsection{On the existence of modelled geodesics issuing from singular points}\label{sec:modelled.geod}

The main purpose of this section is to prove Proposition \ref{prop:admissible.arc.2}, concerning the existence of 
modelled geodesics issuing from singular points. 
This will play a key role in the analysis performed in Sections \ref{sec.behaviour}-\ref{sec.extendability}.

\medskip
As a first step, we recall that in all the standard models, the coordinate vector field $\partial_y$ is Killing (see Remark \ref{rem.proj.v.f.standard}). Hence, for any geodesic 
$$
\gast(t)=(x(t),y(t))
$$ 
of a standard model, 
the quantities $g\left(\dgast,\partial_y\right)$ and $g(\dgast,\dgast)$ are constant along $\gast$. 
Thus, we have that
\begin{equation}\label{syst.principal}
\gst_{22}(x)\,\dot y = C_1,
\qquad
\gst_{11}(x)\,\dot x^{\,2} + \gst_{22}(x)\,\dot y^{\,2} = C_2, \quad C_i\in\R\,.
\end{equation}
If $C_1=0$, then $\dot y\equiv 0$ and the geodesic is horizontal, i.e.\ $y=\mathrm{const}$.  In the case $C_1\neq 0$, eliminating $\dot x$ and $\dot y$ from \eqref{syst.principal}, $\dot x\neq 0$, yields
\begin{equation*}
\gst_{11}(x)+\gst_{22}(x)\,(y')^{2}
=
C\,\frac{\gst_{11}(x)\,\gst_{22}(x)}{C\,\gst_{22}(x)-1}\,,
\quad 
C:=\frac{C_2}{C_1^{2}},
\quad
y'=\frac{\dot y}{\dot x}\,.
\end{equation*}
As a consequence of the first equation of \eqref{syst.principal} we have the following lemma.
\begin{Lemma}\label{lemma.self.inters}
Let $\gast(t)=(x(t),y(t))$ be a geodesic of a standard model $(\mathcal{M},\gst)$. Then the function $y(t)$ is either a constant or strictly monotone. In particular, any geodesic of a standard model is not self-intersecting.
\end{Lemma}
%
%
%
A direct computation shows another property of standard models $(\mathcal{M},\gst)$, contained in the following lemma.
\begin{Lemma}\label{lemma.iso}
Any isometry between open subsets $V$ and $V'$ of the same standard model $(\mathcal{M},\gst)$ is of the form
$$
(x,y)\in V\to\big(x,\epsilon y+k\big)\in V'\,,\quad \epsilon\in \{-1,1\}\,, \,\,k\in\R\,.
$$
In particular, it sends bounded open subsets into  bounded open subsets (in the Euclidean sense).
\end{Lemma}
%
\subsubsection{A first look at the geometry of regular and singular points}\label{sec:first.look}

In this section we  record a simple observation on  the arrangement of regular points corresponding to different standard models. To this purpose, and for practical reasons, we introduce the following sets:
$$
M_{\mathrm{reg}}:=\text{Regular locus of $M$}\,,
$$
$$
M_{\mathcal M}:= \{ \text{points of $M_{\mathrm{reg}}$ locally isometric to $\mathcal{M}$}\}\,,\,\,\text{for each standard model $(\mathcal M,\gst)\,,$}
$$ 
$$
M_{\mathrm{cc}}:= \{ \text{points of $M_{\mathrm{reg}}$ admitting a neighbourhood of constant curvature}\}.
$$
Simple observations lead to the following lemma.
\begin{Lemma}\label{lem:regular_boundary}
Let $(M,g)$ be a $2$-dimensional pseudo-Riemannian manifold with $\dim \mathfrak{p}(g)\ge 2$. 
Then all sets $M_{\mathcal M}$ and $M_{\mathrm{cc}}$ are open and  pairwise disjoint.
Moreover, whenever
$
\partial A \cap \partial B \neq \emptyset,
$
with $A,B \in \{ M_{\mathrm{cc}} \} \cup\bigcup_{\mathcal M} \{ M_{\mathcal M}\}$ and $A \neq B$,
such intersection consists only of singular points. In particular,  $M_{\mathcal M}$ and $M_{\mathrm{cc}}$ form a partition of  $M_{\mathrm{reg}}$.
\end{Lemma}
%
%
%
%
%
\begin{Rem}\label{lem:several.reg}
 If $M$ contains regular points corresponding to at least two different standard models (or to a standard model and constant curvature) and has no singular points, then $M = M_{\mathrm{reg}}$ would be the disjoint union of at least two non-empty open subsets. This contradicts the connectedness of $M$. Therefore, $M$ must admit singular points.
\end{Rem}

\subsubsection{Extension of isometries along projective vector fields}\label{sec.extension.iso}
We now focus our attention on a general mechanism that allows us to propagate isometries along the flow of projective vector fields.  The results contained in the subsequent Lemmas \ref{lemma.tricky} and \ref{lemma.isometry.1} make this principle precise and will serve as a basic tool for the proof of Propositions \ref{prop:admissible.arc} and \ref{prop:admissible.arc.2}.
\begin{Lemma}\label{lemma.tricky}
Let $(M,g)$ and $(M',g')$ be $2$-dimensional pseudo-Riemannian manifolds, and 
$U \subset M$, $U' \subset M'$ be open subsets. Let $X$ and $X'$ be projective
vector fields on $(M,g)$ and $(M',g')$, respectively, and denote by $\phi_t$ and
$\phi'_t$ their local flows.  
Assume that
$$
\varphi :  (U,g) \to (U',g')
$$
is an isometry (resp. a surjective local isometry\,\footnote{By a local isometry we mean a smooth map $f$ between pseudo-Riemannian manifolds $M$ and $N$ with the property that each point of $M$ possesses a neighbourhood $U$ such that $f|_U$ is an isometry. Of course, this does not guarantee that $f$ is injective.}) such that $\varphi_*(X|_U)=X'|_{U'}$. Then the map
$$
\phi'_{t_0} \circ \varphi \circ \phi_{-t_0} :  \phi_{t_0}(U) \to \phi'_{t_0}(U')
$$
is an isometry (resp. a surjective local isometry) sending $X$ to $X'$, for each $t_0$ for which it is well defined.
\end{Lemma}
\begin{proof} 
We shall prove the lemma in the case where $\varphi$ is an isometry; the case where it is a surjective local isometry can be handled with minor straightforward changes.

\smallskip\noindent
If $X\equiv 0$, the statement is trivial. Otherwise, in view of Corollary \ref{cor.proj.field.vanishing}, $X$ cannot vanish identically on nonempty open sets.
Therefore, $X\neq 0$ on a dense open subset of $M$.
Since the statement concerns an identity between smooth tensor fields, precisely $g=(\phi'_{t_0} \circ \varphi  \circ \phi_{-t_0})^*(g')$, it suffices to prove it
in neighbourhoods of points $p\in U$ where $X(p)\neq 0$. Then,
let $p\in U$ and set $p':=\varphi(p)$. Consider the trajectories
$$
\gamma(t):=\phi_t(p)\,,\quad \gamma'(t):=\phi'_t(p'),
$$
that we assume defined on an interval containing $[0,t_0]$.
Choose a smooth curve $c:(-\delta,\delta)\to M$ with $c(0)=p$ and transverse to the trajectory
$\gamma$ at $p$. For $\delta>0$ sufficiently small, the map
$$
C:(-\delta,\delta)\times[0,t_0]\to M,\quad C(s,t):=\phi_t(c(s)),
$$
is well defined and it turns out to be a diffeomorphism onto a tubular neighbourhood of
$\gamma([0,t_0])$.
In particular, $(s,t)$ are coordinates on this neighbourhood. Now, let $c'(s):=\varphi(c(s))$. Similarly,  shrinking $\delta$ if necessary, the map
$$
C':(-\delta,\delta)\times[0,t_0]\to M',\quad C'(s,t):=\phi'_t(c'(s)),
$$
is a diffeomorphism onto a tubular neighbourhood of $\gamma'([0,t_0])$. We identify these neighbourhoods by declaring that points with the same coordinates $(s,t)$ correspond via the map
\begin{equation}\label{eq:Psi.app}
\Psi:=C'\circ C^{-1}.
\end{equation}
By construction, for any $(s,t)$ in the domain of the coordinates, we have
$
\Psi\big(\phi_t(c(s))\big)=\phi'_t\big(\varphi(c(s))\big).
$
Moreover, since $\varphi_*(X|_U)=X'|_{U'}$, it follows that
$
\varphi\circ\phi_t=\phi'_t\circ\varphi
$
whenever both sides are defined.
Choosing $\delta>0$ small enough so that $\phi_t(c(s))\in U$ for $t$ in a neighbourhood
of $0$, it follows that $\Psi=\varphi$ on a neighbourhood $\mathcal{T}$
containing the curve $\{t=0\}$.
Similarly, shrinking $\delta$ if necessary, we have
\begin{equation}\label{eq.Psi}
\Psi=\phi'_{t_0}\circ\varphi\circ\phi_{-t_0}
\end{equation}
on a neighbourhood containing the curve $\{t=t_0\}$.
In these coordinates we have $X=\tfrac{\partial}{\partial t}$ and $X'=\tfrac{\partial}{\partial t}$.
Since $X$ and $X'$ are projective with respect to $g$ and $g'$, respectively, the coefficients $F^0,\mathellipsis,F^3$ and $F'^{0},\mathellipsis,F'^3$ of the associated projective connections (cf. Equation \eqref{eq.proj.conn.associated}) do not depend on the coordinate $t$.
Since $\Psi^*(g')=g$ in the neighbourhood $\mathcal{T}$, the associated projective connections coincide there, and hence $F^i=F'^i$ on such neighbourhood. By $t$-independence, it follows that
$$
F^i(s)=F'^i(s)
$$
on the whole tubular neighbourhood of $\gamma([0,t_0])$ where $\Psi$ given by \eqref{eq:Psi.app} was defined. 
Then $a=  \det(g)^{-\frac23}g$ and $a' =  \det(g')^{-\frac23}g'$ (cf. \eqref{a.from.g}) solve the same system \eqref{lin1} and coincide on a neighbourhood. By Proposition \ref{prop:solutions}, they coincide on the whole tubular neighbourhood. Therefore $g=\Psi^*(g')$ in a neighbourhood containing the curve $\{t=t_0\}$, with $\Psi$ given by \eqref{eq.Psi}. The proposition is proved. 
\end{proof}
\begin{Rem}\label{rem:isometry.ext} 
Under the assumptions of Lemma \ref{lemma.tricky}, the  isometries (resp. surjective local isometries)
$$
\varphi_t:=\phi'_t \circ \varphi \circ \phi_{-t} : \phi_t(U) \to \phi'_t(U')
$$
are compatible in the sense  that, for different values of $t$, they coincide on the intersections of their domains, whenever they are well defined.
Indeed, the condition $\varphi_{t*}(X) = X'$ implies that $\varphi_t$ intertwines the
local flows of $X$ and $X'$.
Consequently, for any $t_0$ such that all these compositions are defined for
$t \in [0,t_0]$, they glue together to define a surjective local isometry
\begin{equation*}
\widetilde{\varphi}:\bigcup_{t\in[0,t_0]} \phi_t(U) \to \bigcup_{t\in[0,t_0]} \phi'_t(U')
\end{equation*}
still sending $X$ to $X'$ and extending $\varphi$ (i.e.  $\widetilde{\varphi}\,|_{U}=\varphi$). In particular, $\widetilde\varphi$ is a covering map.
\end{Rem} 

\begin{Lemma}\label{lemma.isometry.1}  
Let $(M,g)$ be a $2$-dimensional pseudo-Riemannian manifold with $\dim\mathfrak{p}(g)\geq 2$. Assume that every regular point admits a neighbourhood isometric to an open subset of one of the standard models. Let $\gamma:[0,1]\to M$ be a curve such that points $\gamma(t)$ are regular for any $t>0$. Assume that there exist two
projective vector fields in $\mathfrak{p}(g)$ which are non-proportional at $\gamma(t)$  for every $t>0$. Then there exists a neighbourhood of $\gamma((0,1])$ which is locally isometric, via a covering map, to an open subset of one of the standard models $(\mathcal{M},\gst)$.
\end{Lemma}
\begin{proof}
We shall prove the lemma in the case where the curve $\gamma$ is not self-intersecting in order to provide a clearer and more concise proof; the case where it is self-intersecting can be handled with minor straightforward changes.

\smallskip\noindent
For each $t\in(0,1]$, the point $\gamma(t)$ is regular. Hence, there exists an open
neighbourhood of $\gamma(t)$ which is  isometric to an open subset of a
standard model  $(\mathcal M,\gst)$. 
In particular, there exists an isometry
$$
\varphi_1 :  U_1 \to V_1
$$
from a neighbourhood  $U_1\subset M$ of $\gamma(1)$ into an open subset $V_1 \subset \mathcal M$. If necessary, we replace $U_1$ by a smaller neighbourhood, that we continue denoting by $U_1$, such that $\overline{U_1}$ and $\overline{\varphi(U_1)}$ are compact, respectively, in $M$ and $\mathcal M$
and, by considering  that we are assuming that $\gamma$ has no self-intersections,
$\gamma^{-1}(U_1)\subset(0,1]$ is an interval.

\medskip\noindent
If $\gamma(0)\in\partial U_1$, then
$U_1$ is already a neighbourhood of $\gamma((0,1])$ and there is nothing to prove.

\medskip\noindent
Otherwise, let $t_1\in(0,1]$ be such that $\gamma^{-1}(U_1)=(t_1,1]$, namely $\gamma(t_1)\in\partial U_1$. 
By hypothesis, there exist two projective vector fields in $\mathfrak p(g)$ which
are non-proportional at $\gamma(t_1)$. Therefore, we may choose a projective
vector field $X\in\mathfrak p(g)$ such that its local flow $\phi_\tau$ enters $U_1$ at $\gamma(t_1)$. For instance, we may consider $X\in\mathfrak p(g)$ such that $X(\gamma(t_1))=\dot \gamma(t_1)$.
Hence, there exists $\varepsilon>0$ for which
$$
\phi_\tau(\gamma(t_1)) \in U_1 \,\,\text{for all } \tau \in (0,\varepsilon).
$$
Restricting $X$ to $U_1$ and pushing it forward by the  isometry
$\varphi_1 : U_1 \to V_1$, we obtain a projective vector field on $V_1$.
Any projective vector field defined on an open subset of a standard model $(\mathcal{M},\gst)$ is the
restriction of a globally defined projective vector field, cf. Remark \ref{rem.proj.v.f.standard}.
Denote by $Y$ the corresponding projective vector field on the whole standard model
$\mathcal M \subseteq \mathbb{R}^2$, so that
$
Y|_{V_1} = (\varphi_1)_* \,X|_{U_1}.
$
Since $\overline{U_1}$ and $\overline{V_1}$ are compact and $X$, $Y$ are global smooth vector fields,
there exists $\varepsilon_0>0$ such that the local flows $\phi_\tau$ of $X$ and
$\phi'_\tau$ of $Y$ are defined for all points in $U_1$ and $V_1$, respectively, and
all $\tau\in(-\varepsilon_0,\varepsilon_0)$. Thus, the assumptions of Lemma \ref{lemma.tricky} are satisfied.
Applying Lemma \ref{lemma.tricky} together with Remark \ref{rem:isometry.ext}, we obtain a
local isometry
$$
\varphi_2:  U_2\to V_2,
$$
where $U_2\subset M$ is an open set  containing $U_1$ and point
$\gamma(t_1)$. If $\partial U_2\cap\gamma=\emptyset$, then $U_2$ is already the neighbourhood we are looking for. In the other case, i.e., if $\partial U_2\cap\gamma=\gamma(t_2)$ for some $t_2\in (0,1]$, if necessary we replace $U_2$ by a smaller neighbourhood of $U_1 \cup \{\gamma[t_1,t_2)\}$, that we continue denoting by $U_2$, such that its closure is compact. Furthermore, $\overline{\varphi_2(U_2)}$ is bounded and then compact. Indeed, by contradiction, let $\overline{\varphi_2(U_2)}$ be unbounded. Since $\gamma(t_2)$ is a regular point, there exists an isometry $\psi$ between a (suitable small) neighbourhood of $\gamma(t_2)$  into a bounded open subset of the standard model $(\mathcal{M},\gst)$. Such open subset would contain a bounded open subset of $(\mathcal{M},\gst)$ that would be isometric, via $\varphi_2\circ\psi^{-1}$, to an unbounded open subset of  $(\mathcal{M},\gst)$. 
This is not possible in view of Lemma \ref{lemma.iso}.


\smallskip\noindent
Iterating the above argument, we construct two increasing sequences of open sets
$$
U_1 \subset U_2 \subset \cdots \subset M,
\qquad\quad
V_1 \subset V_2 \subset \cdots \subset \mathcal M,
$$
together with local isometries $\varphi_k : U_k \to V_k$ such that
$\varphi_{k+1}|_{U_k}=\varphi_k$ for every $k$. Indeed, whenever
$$
t_k := \inf\{\, t\in(0,1] \mid \gamma((t,1]) \subset U_k \,\} \, >0,
$$
the above construction can be repeated at the point $\gamma(t_k)$. By the hypothesis that at $\gamma(t_k)$ there exist two non-proportional projective vector fields, we may choose $X\in\mathfrak p(g)$ pointing into $U_k$ at $\gamma(t_k)$. Applying Lemma \ref{lemma.tricky}, we obtain a local
isometry $\varphi_{k+1}:  U_{k+1}\to V_{k+1}$, with $U_{k+1}$ containing $U_k$
and $\gamma(t_k)$, extending $\varphi_k$. 
By compatibility, all local isometries obtained through the above procedure glue together to a maximal local isometry defined on a neighbourhood of $\gamma((\tilde t,1])$, for some $\tilde t\geq0$. We claim that $\tilde t=0$. Indeed, if $\tilde t>0$, then $\gamma(\tilde t)$ is regular, and the same extension argument can be applied there, extending the local isometry beyond $\gamma(\tilde t)$, in contradiction with its maximality.
 By Remark \ref{rem:isometry.ext}, such local isometry is a covering map. This concludes the proof.
\end{proof}
\begin{Rem}\label{rem.no.self.inters}
The assumption that the projective vector fields are non-proportional at every point of the curve is made for simplicity of  exposition: in the proof, non-proportionality is actually required only
at the set of points where the (local) isometry is extended.
\end{Rem}
\subsubsection{An obstruction to the applicability of Lemma \ref{lemma.isometry.1}: the non-transitivity locus}\label{sec.obstruction}
The condition that the projective vector fields be non-proportional along the curve
$\gamma$ in Lemma \ref{lemma.isometry.1} is not an artificial requirement: it does not follow from the mere regularity of the points of $\gamma$. Indeed, although at every point of a standard model there exist two projective vector fields which are non-proportional at that point (see Remark \ref{rem.proj.v.f.standard}), this does not imply the existence of a pair of globally defined projective vector fields which are non-proportional at every point.
The obstruction arises precisely for $2$-dimensional pseudo-Riemannian manifolds admitting a $2$-dimensional  projective algebra  and having regular points of type \ref{case2a}, \ref{case2b} or \ref{case2c}. Indeed, only in this situation two non-proportional projective vector fields may become pointwise linearly dependent along curves of regular points. Such phenomenon is illustrated in the next example and is more systematically analysed in the subsequent lemma.
\begin{Ex} \label{ex.1}
We exhibit a $2$-dimensional subalgebra of the projective Lie algebra of the standard model $\gst$
of type \ref{case2c} whose generators become collinear along curves of the domain of the metric.\\
Assume $\varepsilon_1\varepsilon_2=1$, so that a basis of $\mathfrak p(\gst)$, see Remark \ref{rem.proj.v.f.standard}, is given by
$$
X_1=\partial_y,\qquad
X_2=-x\cos(y)\,\partial_x+\sin(y)\,\partial_y,\qquad
X_3=x\sin(y)\,\partial_x+\cos(y)\,\partial_y .
$$
Vector fields $X_1+X_2$ and $X_3$ form a $2$-dimensional subalgebra $\mathfrak h(\gst)\subset\mathfrak p(\gst)$ as $[X_1+X_2,X_3]=-(X_1+X_2)$:
%
%
its generators become collinear
precisely along the lines
$y=-\tfrac{\pi}{2}+2k\pi$, $k\in\mathbb Z$,
recalling that we are dealing with the points of the domain of the standard metric \ref{case2c} of Theorem \ref{th.Bryant.Manno.Matveev}.
\end{Ex}
%
For the reasons explained above, we need to introduce the \emph{non-transitivity locus} $\Sigma_N(\mathfrak{h})$ of a Lie subalgebra $\mathfrak{h}$  of vector fields on a $2$-dimensional manifold $N$ as follows:
\begin{equation}\label{eq:non.trans.locus}
\Sigma_N(\mathfrak{h}):=\{p\in N\,\,|\,\,\dim\operatorname{span}\mathfrak{h}_p<2\}\,.
\end{equation}
%
Since we will study this notion in the context of $2$-dimensional subalgebras of $\mathfrak{p}(g)$, we first establish their existence in the next lemma.
\begin{Lemma}\label{lem:exist.hg}
Let $(M,g)$ be a $2$-dimensional pseudo-Riemannian manifold with $\dim\mathfrak{p}(g)\geq 2$.  Assume that there exists a regular point admitting a neighbourhood isometric to an open subset of one of the standard models. Then there exists a $2$-dimensional Lie subalgebra $\mathfrak{h}(g)$ of $\mathfrak{p}(g)$.
\end{Lemma}
\begin{proof}
Under our assumptions, one has $\dim\mathfrak{p}(g) \in \{2,3\}$. 
If $\dim\mathfrak{p}(g)=2$, the statement is trivial.

Assume now that $\dim\mathfrak{p}(g)=3$. Then there exists an open subset $U \subseteq M$ which is isometric to an open subset $V$ of one of the standard models $(\mathcal M,\gst)$ of type \ref{case2a}, \ref{case2b}, or \ref{case2c}. Let $\phi: U \to V$ be such an isometry. 
Since the Lie algebra $\mathfrak{p}(\gst)$ contains a $2$-dimensional Lie subalgebra $\mathfrak{h}(\gst)$ (cf.\ Remark \ref{rem.proj.v.f.standard}), pushing forward its generators via $\phi^{-1}$ we obtain a $2$-dimensional Lie algebra of projective vector fields on $U$. 
Since $\phi_*\,\mathfrak{p}(g|_U)=\mathfrak{p}(\gst|_V)$ and, by Remark \ref{rem.proj.v.f.standard},
every projective vector field on $\gst|_V$ is the restriction of a projective vector field on $(\mathcal M,\gst)$, we have
$$
\dim\mathfrak{p}(g|_U)=\dim\mathfrak{p}(\gst|_V)=\dim \mathfrak{p}(\gst)=3.
$$
Since $\dim\mathfrak{p}(g)=3$ and by considering that  $\mathfrak{p}(g)|_U\subseteq \mathfrak{p}(g|_U)$, these vector fields extend uniquely to global projective vector fields $Y$ and $Z$ on $M$, i.e. $Y, Z \in \mathfrak{p}(g)$. Choosing the generators appropriately, we may assume that
$[Y,Z]\big|_U = Y\big|_U.$
Then the vector field $[Y,Z] - Y$ is a projective vector field on $M$ vanishing on $U$. By Corollary \ref{cor.proj.field.vanishing}, it follows that $[Y,Z] = Y$ on $M$. Therefore, $Y$ and $Z$ generate a $2$-dimensional Lie subalgebra $\mathfrak{h}(g) \subset \mathfrak{p}(g)$.
\end{proof}
Now we are in position to give a more precise description of the non-transitivity locus in our context.
\begin{Lemma}\label{lem:noopen.collinearity}
Let $(M,g)$ be a $2$-dimensional pseudo-Riemannian manifold with $\dim\mathfrak p(g)\geq 2$.
Assume that every regular point admits a neighbourhood isometric to an open subset of one of the standard models.
Let $\mathfrak h(g)\subseteq\mathfrak p(g)$ be an arbitrary $2$-dimensional Lie subalgebra of $\mathfrak p(g)$ (it exists by Lemma \ref{lem:exist.hg}).
Then the non-transitivity locus $\Sigma_{M_\mathrm{reg}}(\mathfrak h(g))$ of $\mathfrak h(g)$ is either empty 
or a union of disjoint smooth curves.
\end{Lemma}
\begin{proof}
Let $\mathfrak h(g)$ be spanned by $X,Y$. Choose a neighbourhood $U\subseteq M$ of an
arbitrary regular point $p\in M$ together with an isometry
$$
\varphi:(U,g)\to (V,\gst)\,.
$$
Then, in view of  Remark \ref{rem.proj.v.f.standard}, there exist $X_{\mathrm{std}},Y_{\mathrm{std}}\in\mathfrak p(\gst)$ such that their restrictions to $V$ coincide with the push-forwards of $X|_U$ and $Y|_U$ via $\varphi$, respectively, i.e.,
$X_{\mathrm{std}}|_V=\varphi_*(X|_U)$, $Y_{\mathrm{std}}|_V=\varphi_*(Y|_U).$
By considering again Remark \ref{rem.proj.v.f.standard}, a direct computation shows that $X_{\mathrm{std}}\wedge Y_{\mathrm{std}}\equiv 0$ on a non-empty open set $W\subseteq V$ if and only if $X_{\mathrm{std}}$ and $Y_{\mathrm{std}}$ are proportional on $W$ by a constant.
In this case, $X$ and $Y$ are proportional by the same constant on the open set $\varphi^{-1}(W)\subseteq M$, and hence on all of $M$ by
Proposition \ref{homo}, contradicting that we are assuming $\dim\mathfrak h=2$. Therefore $X_{\mathrm{std}}\wedge Y_{\mathrm{std}}$ cannot vanish on any non-empty open subset of
$V$.

When $(V,\gst)$ is of type $\ref{case1a}$, $\ref{case1b}$ or $\ref{case1c}$, by
Remark \ref{rem.proj.v.f.standard}, 
if $X_{\mathrm{std}}\wedge Y_{\mathrm{std}}$ vanishes at a point, then $X_{\mathrm{std}}$ and $Y_{\mathrm{std}}$ are proportional by a constant, that is a contradiction.
Hence
$$
\Sigma_{M_\mathrm{reg}}(\mathfrak h(g))\cap U
=\varphi^{-1}\big(\{X_{\mathrm{std}}\wedge Y_{\mathrm{std}}=0\}\big)=\emptyset.
$$
Assume now that $(V,\gst)$ is of type $\ref{case2a}$, $\ref{case2b}$ or $\ref{case2c}$ with
$\varepsilon_1\varepsilon_2=-1$. A long computation shows that, when non-empty, the
non-transitivity locus $\Sigma_{\mathcal M}\big({\mathfrak h(\gst)}\big)$ of an arbitrary $2$-dimensional Lie subalgebra $\mathfrak h(\gst)$ of $\mathfrak p(\gst)$ is a single straight line of the form
$
y=y_0,
$
where the value $y_0\in\R$ depends on the considered Lie subalgebra $\mathfrak h(\gst)$. 
In the remaining case
of the model of type $\ref{case2c}$, i.e., when $\varepsilon_1\varepsilon_2=1$, the corresponding
non-transitivity locus $\Sigma_{\mathcal M}\big({\mathfrak h(\gst)}\big)$, when non-empty, is a discrete family of parallel lines, namely
either
$y=y_0+2h\pi$ or $y=y_0+h\pi$, $ h\in\Z$,
where the value of $y_0$ and the minimal period are determined by the specific choice of
the $2$-dimensional Lie subalgebra  $\mathfrak h(\gst)$. Therefore, in all these cases,
$\Sigma_{M_\mathrm{reg}}(\mathfrak h(g))\cap U$ is either empty or a union of disjoint smooth curves in $U$.
Since $p$ was arbitrary, the lemma follows.
\end{proof}
\begin{Rem}\label{rem:Sigma.notempty}
As already shown in the proof of Lemma \ref{lem:noopen.collinearity}, the condition
$\Sigma_{M_\mathrm{reg}}(\mathfrak h(g))=\emptyset$ holds whenever the regular points are of type $\ref{case1a}$,
$\ref{case1b}$ or $\ref{case1c}$. 
Moreover, if $\dim\mathfrak{p}(g)=3$ (hence there exist regular points  of type $\ref{case2a}$, $\ref{case2b}$ or $\ref{case2c}$), then  there exists a (nontrivial) linear combination of generators of $\mathfrak p(g)$ whose push-forward by an isometry is a Killing vector field of the corresponding standard model. Therefore, by Proposition \ref{homo}, such linear combination turns out to be a (global) Killing vector field on $M$.
Since in the standard models of type $\ref{case2a}$, $\ref{case2b}$ and $\ref{case2c}$ any $2$-dimensional subalgebra $\mathfrak h(\gst)$ containing the Killing field satisfies  $\Sigma_{\mathcal M}\big(\mathfrak h(\gst)\big)=\emptyset$,  the condition $\Sigma_{M_{\mathrm{reg}}}(\mathfrak h(g))\neq\emptyset$ implies that  $\dim \mathfrak{p}(g)=2$ and that regular points are of type $\ref{case2a}$, $\ref{case2b}$ or $\ref{case2c}$.
\end{Rem}
%
%
%
As we said in the beginning of this section, our target is to prove Proposition \ref{prop:admissible.arc.2}, that concerns geodesics of $(M,g)$ with $\dim\mathfrak{p}(g)\geq 2$ rather than general curves.
In order to do it, we need a last technical lemma, concerning the behaviour of geodesics of standard models $(\mathcal{M}, \gst)$ near the non-transitivity locus $\Sigma_{\mathcal M}(\mathfrak h (\gst))$ of a $2$-dimensional Lie subalgebra $\mathfrak h (\gst) \subseteq \mathfrak{p} (\gst)$ (recall the definition given by \eqref{eq:non.trans.locus}). 
Indeed, the non-transitivity locus is the only possible place where the
direct application of Lemma \ref{lemma.isometry.1}  may fail.  
The next lemma shows that, along geodesics of the standard models,
this possible failure is either persistent, when the geodesic lies in a
component of the non-transitivity locus, or occurs only at isolated points.

\begin{Lemma}\label{lem:geodesic.Sigma}
Let $(\mathcal M,\gst)$ be a standard model, and let
$\mathfrak h(\gst)\subseteq\mathfrak p(\gst)$ be a $2$-dimensional Lie subalgebra such that $\Sigma_{\mathcal M}\big(\mathfrak h(\gst)\big)\neq \emptyset$.  Let  $\gast:(0,1)\to \mathcal M$ be a geodesic. Then the following properties hold:
\begin{enumerate}
\item\label{lem:geodesic.Sigma:tangent} If  $\gast$  is tangent to a connected
component of $\Sigma_{\mathcal M}\big(\mathfrak h(\gst)\big)$ at some point, then $\gast$ is entirely contained
in that component.
\item\label{lem:geodesic.Sigma:transverse} If $\gast$ intersects  $\Sigma_{\mathcal M}\big(\mathfrak h(\gst)\big)$
transversely at some point, then it intersects each connected component of $\Sigma_{\mathcal M}\big(\mathfrak h(\gst)\big)$ at most once. In particular,
$
\gast\cap \Sigma_{\mathcal M}\big(\mathfrak h(\gst)\big)
$
is a discrete set.
\end{enumerate}
\end{Lemma}
\begin{proof}
By the explicit description of the non-transitivity locus $\Sigma_{\mathcal M}\big(\mathfrak h(\gst)\big)$ 
(cf. proof of Lemma \ref{lem:noopen.collinearity}), whenever $\Sigma_{\mathcal M}\big(\mathfrak h(\gst)\big)\neq\emptyset$, it is a union of straight lines. More precisely, for $\ref{case2a}$, $\ref{case2b}$, and for $\ref{case2c}$ with $\varepsilon_1\varepsilon_2=-1$, one has
$
\Sigma_{\mathcal M}\big(\mathfrak h(\gst)\big)=\{y=y_0\}$,
whereas for $\ref{case2c}$ with $\varepsilon_1\varepsilon_2=1$ one has
$
\Sigma_{\mathcal M}\big(\mathfrak h(\gst)\big)=\{\,y=y_0+mT\,\}_{m\in\mathbb Z}$,
with $T\in\{\pi,2\pi\}$ depending on the considered Lie subalgebra $\mathfrak h(\gst)$.

If $\gast(t)=\big( x(t), y(t)\big)$ is tangent to $\Sigma_{\mathcal M}\big(\mathfrak h(\gst)\big)$ at some point $t_0$,
then $\dot y(t_0)=0$. By the first integral $g^{\rm std}_{22}(x(t))\dot y(t)=0$, it
follows that $y$ is constant along $\gast$. Therefore, $\gast$ is entirely contained in the same
connected component of $\Sigma_{\mathcal M}\big(\mathfrak h(\gst)\big)$.
Assume now that $\gast$ intersects $\Sigma_{\mathcal M}\big(\mathfrak h(\gst)\big)$ transversely. By Lemma \ref{lemma.self.inters}, whenever $y$ is not constant, it is strictly monotone along $\gast$.
Hence, $\gast$ can intersect each connected component of $\Sigma_{\mathcal M}\big(\mathfrak h(\gst)\big)$ at most once. Therefore, in the cases $\ref{case2a}$, $\ref{case2b}$, and $\ref{case2c}$ with
$\varepsilon_1\varepsilon_2=-1$,  $\gast\cap\Sigma_{\mathcal M}\big(\mathfrak h(\gst)\big)$ consists of a single point.
Moreover, in  the remaining case $\ref{case2c}$, i.e. when $\varepsilon_1\varepsilon_2=1$, since the set of values $\{y_0+mT\}_{m\in\mathbb Z}$ is discrete, the set $\gast\cap\Sigma_{\mathcal M}\big(\mathfrak h(\gst)\big)$  is discrete.
\end{proof}

\subsubsection{Proof of the existence of modelled geodesics issuing from singular points}\label{sec.prop.3.4}

\begin{Prop}\label{prop:admissible.arc}
Let $(M,g)$ be a $2$-dimensional pseudo-Riemannian manifold with $\dim\mathfrak{p}(g)\ge 2$. Assume that every regular point admits a neighbourhood isometric to an open subset of one of the standard models.
Let $\gamma:[0,1]\to M$ be a geodesic such that $\gamma(0)$ is singular and $\gamma(t)$ are regular points for $t\in (0,1]$.
Then $\gamma$ is a modelled geodesic.
\end{Prop}
\begin{proof}

Since $\dim\mathfrak p(g)\geq 2$, there exists a $2$-dimensional Lie
subalgebra $\mathfrak h(g)\subseteq\mathfrak p(g)$
(see Lemma \ref{lem:exist.hg}). Let $\{X,Y\}$ be a basis of
$\mathfrak h(g)$. If
$(X\wedge Y)(\gamma(t))\neq 0$
for every $t\in(0,1]$,
then Lemma \ref{lemma.isometry.1} applies and provides a neighbourhood
of $\gamma((0,1])$ which is locally isometric, via a covering map, to
an open subset of a standard model. Since local isometries map
geodesics to geodesics and geodesics in the standard models have no
self-intersections (see Lemma \ref{lemma.self.inters}), by shrinking
the neighbourhood if necessary we may assume that its image is simply
connected. The covering map is then an isometry onto its image, and
we are done.
Assume now that
$$
T:=\left\{
t\in(0,1]\mid
\gamma(t)\in
\Sigma_{M_{\mathrm{reg}}}\bigl(\mathfrak h(g)\bigr)
\right\}\neq\emptyset.$$

Suppose first that $\gamma$ is tangent to
$\Sigma_{M_{\mathrm{reg}}}\bigl(\mathfrak h(g)\bigr)$ at some point.
In view of Lemma \ref{lem:geodesic.Sigma}
(\ref{lem:geodesic.Sigma:tangent}), we obtain
$
\gamma((0,1])
\subset
\Sigma_{M_{\mathrm{reg}}}\bigl(\mathfrak h(g)\bigr)$.
Along this set the distribution generated by $\mathfrak h(g)$ has
rank one (cf. Remark \ref{rem.proj.v.f.standard}) and is tangent to
$\Sigma_{M_{\mathrm{reg}}}\bigl(\mathfrak h(g)\bigr)$, since the rank of this distribution is preserved by the local flows of $\mathfrak h(g)$. Hence, at every point of $\gamma((0,1])$, one can
choose $Z\neq 0 \in\mathfrak h(g)$ such that
$Z(\gamma(t))$
is parallel to $\dot\gamma(t)$.
Starting from a neighbourhood of $\gamma(1)$ isometric to an open
subset of the corresponding standard model and suitably choosing the
sign of $Z$ whenever necessary, Lemma \ref{lemma.tricky}, together
with Remark \ref{rem:isometry.ext}, allows us to extend the isometry
successively along $\gamma$, as in the proof of
Lemma \ref{lemma.isometry.1}. The resulting map is a covering local
isometry. As above, restricting it over a sufficiently small simply
connected neighbourhood of the corresponding standard geodesic
yields an isometry onto its image. Thus $\gamma$ is modelled.

\smallskip

From now on, we assume that every intersection of $\gamma$ with
$\Sigma_{M_{\mathrm{reg}}}\bigl(\mathfrak h(g)\bigr)$ is transverse.
By point \ref{lem:geodesic.Sigma:transverse} of Lemma \ref{lem:geodesic.Sigma}, the set $T$ is discrete.
Moreover, $T$ is closed in $(0,1]$, being the zero set of
$(X\wedge Y)\circ\gamma$. It follows that
$
T\cap[\delta,1]
$
is finite for every $\delta>0$.
By Lemma \ref{lem:regular_boundary}, all points of $\gamma((0,1])$
are locally isometric to open subsets of the same standard model
$(\mathcal M,\gst)$. Starting from a standard neighbourhood of
$\gamma(1)$, we extend the corresponding isometry along $\gamma$ by
the procedure used in the proof of Lemma \ref{lemma.isometry.1}
(see also Remark \ref{rem.no.self.inters}). Indeed, whenever the
boundary of its domain meets $\gamma$ at a point $\gamma(t)$ with
$t>0$ and $t\notin T$, the vector fields $X$ and $Y$ are
non-proportional at $\gamma(t)$, and the same extension argument
applies.
At every such step, Remark \ref{rem:isometry.ext} gives a covering
local isometry. Since the image of the relevant portion of $\gamma$
is a geodesic of $(\mathcal M,\gst)$ without self-intersections, we
may restrict the covering to the component lying over a sufficiently
small simply connected tubular neighbourhood of this geodesic
portion and containing the previously constructed domain. On this component the covering is an isometry
$
\varphi:(U,g)\to(V,\gst)$.
Therefore, either $\partial U$ meets $\gamma$ only at $\gamma(0)$,
in which case we are done, or $\partial U$ meets $\gamma$ at a
regular point $\gamma(t_0)$ with $t_0\in T$.
Since $\gamma(t_0)$ is regular, there exist a neighbourhood $U_0$ of
$\gamma(t_0)$ and an isometry
$
\psi:(U_0,g)\to(\mathcal M,\gst)$.
After restricting to a sufficiently small open subset $W$ of
$U\cap U_0$, the transition map
$
A:=\psi|_W\circ(\varphi|_W)^{-1}$
is an isometry between open subsets of $(\mathcal M,\gst)$. By
Lemma \ref{lemma.iso},
$
A(x,y)=(x,\epsilon y+k)$,
with $\epsilon\in\{-1,1\},\, k\in\R$.
The same formula defines an isometry of the whole standard model.
Replacing $\psi$ by $A^{-1}\circ\psi$, we may therefore assume that
$
\psi=\varphi
$ on $W$.
Since the corresponding  geodesic in the standard model has no self-intersections,
after shrinking $U$ and $U_0$ to suitable tubular neighbourhoods we
may assume that $U\cap U_0$ is simply connected and
$\psi=\varphi$ on $U\cap U_0$,
so that the two maps glue to an isometry on $U\cup U_0$, extending
$\varphi$ across $\gamma(t_0)$.
We may now resume the extension along $\gamma$ and repeat the same
argument whenever a point of $T$ is encountered. Since
$T\cap[\delta,1]$ is finite for every $\delta>0$, only finitely many
such gluings are required on each arc $\gamma([\delta,1])$. Choosing
the tubular neighbourhoods successively so that the resulting
isometries extend the previously constructed ones, we obtain
increasing open sets
$
U_1\subset U_2\subset\cdots
$
whose union is a neighbourhood of $\gamma((0,1])$, together with
isometries
$
\varphi_n:(U_n,g)\to (V_n,\gst)$, where $
\varphi_{n+1}|_{U_n}=\varphi_n$.
Their union
$
\varphi:=\bigcup_n\varphi_n
$
is therefore an isometry onto its image. Hence $\gamma$ is a modelled
geodesic.
\end{proof}
%
%
%
We observe that modelled geodesics necessarily arise also in the presence of singular points. More precisely, we have the following proposition.
\begin{Prop}\label{prop:admissible.arc.2}
Let $(M,g)$ be a $2$-dimensional pseudo-Riemannian manifold with $\dim\mathfrak{p}(g)\ge 2$. Let $s\in M$ be a singular point. Then there exists a geodesic issuing from $s$ containing a modelled geodesic arc issuing from a (possibly different) singular point.
\end{Prop}
\begin{proof}
Since $\dim\mathfrak{p}(g)\ge 2$, there exists a $2$-dimensional Lie subalgebra $\mathfrak h(g)\subseteq\mathfrak p(g)$ (see Lemma \ref{lem:exist.hg}). Let $\{X,Y\}$ be a basis of $\mathfrak h(g)$.
By Lemma \ref{lem:noopen.collinearity}, the set $M_{\mathrm{reg}}\setminus \Sigma_{M_{\mathrm{reg}}}(\mathfrak h(g))$ is open and dense in $M$. Hence, we can choose a point $r\in M_{\mathrm{reg}}\setminus \Sigma_{M_{\mathrm{reg}}}(\mathfrak h(g))$ arbitrarily close to $s$. Taking $r$ inside a normal neighbourhood of $s$, there exists a geodesic
$\gamma:[0,1]\to M$ such that $\gamma(0)=s$ and $\gamma(1)=r$. 
Let
$$
T:=\{t\in[0,1]\mid \gamma(t)\ \text{is singular}\}.$$
Since the singular locus is closed, $T$ is closed. Since $T$ is a non-empty compact subset of $[0,1]$, it admits a maximum $\bar{t}:=\max T$.
After restricting to $[\bar{t},1]$ and reparametrizing, we may assume that $\gamma(0)$ is singular and $\gamma(t)$ is regular for every $t\in(0,1]$. By Proposition \ref{prop:admissible.arc}, such a geodesic arc is modelled.
\end{proof}
\begin{Cor}\label{cor:different.types}
Let $(M,g)$ be a $2$-dimensional pseudo-Riemannian
manifold with $\dim \mathfrak p(g)\geq 2$. Assume that $M_{\rm reg}$
contains regular points  either of at least of two different types or of the same type with different
values of the parameters. Then, for each non-empty connected component $A$ of $M_{\rm reg}$, there exists a modelled geodesic arc
$\gamma:[0,1]\to M$ issuing from a singular point such that $\gamma((0,1))\subset A$.
\end{Cor}
\begin{proof}
The existence of singular points follows from Lemma \ref{lem:regular_boundary} and Remark \ref{lem:several.reg}.
The existence of the required modelled geodesic arc follows from
Propositions \ref{prop:admissible.arc} and \ref{prop:admissible.arc.2}.
\end{proof}
\begin{Rem}\label{lem:two.singular}
In the situation described above (see also Figure \ref{fig.mod.geod}), a geodesic issuing from a singular point may meet other singular points. Hence, there can  exist geodesic arcs whose endpoints are singular while all interior points are regular. Any such arc is a modelled geodesic:
this follows from the same argument of Proposition \ref{prop:admissible.arc}. When needed, we will specify that they have \emph{singular endpoints} to distinguish them from those of Proposition \ref{prop:admissible.arc}.
\end{Rem}

\subsection{Behaviour of metric invariants along suspicious geodesics 
}\label{sec.behaviour}

The previous section showed that the study of singular points is closely related to that of modelled geodesic arcs. We now look at the realizations of such arcs in the standard models under modelling isometries. The goal
of this section is to prove that these realizations are necessarily suspicious geodesics.
To this purpose, let $\gamma$ be a modelled geodesic of $(M,g)$ issuing from a singular point and let $\gast$ be its image in a standard model $(\mathcal M, \gst)$.
 The  geodesic $\gast$ is either bounded in $\R^2$ or escapes every compact subset of $\R^2$. Some considerations when it is bounded are in order since, a priori, some accumulation phenomena of points of $\gast$ can appear. A more precise explanation is given below.

\smallskip\noindent
Suppose  that $\gast$ is bounded. If, in addition, there exists a projective vector field whose flow
points from the singular point $\gamma(0)$ towards the neighbourhood of the regular part where the  isometry is defined, then points $\gast(t)$ cannot accumulate at an interior point
of $\mathcal M$. Indeed, if the closure of $\gast$ were contained in $\mathcal M$,
Lemma \ref{lemma.tricky} would allow the  isometry to extend up to $\gamma(0)$, producing a neighbourhood of the singular point that is isometric to a neighbourhood of some standard model, a contradiction.
Hence, in this situation, the singular point $\gamma(0)$ corresponds to a boundary
point of $\mathcal M$ in $\R^2$.
A priori, it remains possible that the singular point $\gamma(0)$ corresponds to an interior accumulation point of the model.
In that case, the limit point must lie in a non-transitivity locus $\Sigma_{\mathcal M}(\mathfrak{h}(\gst))\subset \R^2$.
The analysis of this possibility and its exclusion is carried out in Proposition \ref{lem:no-interior-accumulation}, whose proof requires some technical lemmas contained in the Appendix.

\begin{Prop}\label{lem:no-interior-accumulation}
Let $(M,g)$ be a $2$-dimensional pseudo-Riemannian manifold with $\dim\mathfrak{p}(g)\ge 2$.
Let $\gamma$ be a modelled geodesic in $(M,g)$ issuing from a singular point, and let $\gast$ denote its image in the standard model $(\mathcal M,\gst)$ under an isometry. Then $\gast$ is a suspicious geodesic.
\end{Prop}
\begin{proof}
In order to get our statement, we have to show that, if $\gast(t)$ is bounded as  $t\to 0^+$, then  it tends to $\partial M$, cf. Definition \ref{def:gast}.

Let $\gast(t)=(x(t),y(t))$. Since $\gast$ is assumed to be bounded in $\R^2$ as  $t\to 0^+$, both $x(t)$ and $y(t)$ are also bounded. By Lemma \ref{lemma.self.inters}, $y(t)$ is either constant or strictly monotone. Hence
the limit $\lim_{t\to 0^+}y(t)$ exists, it is finite and it will be denoted by $\bar{y}$.
In view of \eqref{syst.principal}, we have
$$
\dot x(t)^{2}=\frac{\gst_{22}( x(t))\, C_2-C_1^2}{\gst_{22}( x(t))\; \gst_{11} (x(t))}\,.
$$
Since, by assumption, $x(t)$ is bounded on $(0,1]$, in view of the above formula, considering that the denominator is nowhere vanishing on $\R^2$, it follows that $\dot x(t)$ is also bounded on $(0,1]$. Therefore, $\lim_{t\to 0^+}x(t)$ exists, it is finite and it will be denoted by $\bar{x}$.
Thus, 
$$
\lim_{t\to 0^+}\gast(t)=(\bar{x},\bar{y})\in\R^2\,.
$$
Now assume by contradiction that $(\bar{x},\bar{y})\in\mathcal M$.

If points of $\gast$ are of type $\ref{case1a}$, $\ref{case1b}$ or $\ref{case1c}$, then in the corresponding standard model there exist two projective vector fields $X_{\mathrm{std}},Y_{\mathrm{std}}$ which are non-proportional at every point of $\mathcal M$.
In particular, $(X_{\mathrm{std}}\wedge Y_{\mathrm{std}})(\gast(t))\neq 0$ for any $t$.
Pulling back $X_{\mathrm{std}},Y_{\mathrm{std}}$ along the isometry defined near the part of $\gamma$ consisting of regular points,
we obtain two projective vector fields on $M$ which are non-proportional along $\gamma((0,1])$. By continuity, they are non-proportional at $\gamma(0)$ as well.
Lemma \ref{lemma.isometry.1} then yields a neighbourhood of $\gamma(0)$ isometric to a standard model,
contradicting its singularity.

Assume now that points of $\gast$ are of type $\ref{case2a}$, $\ref{case2b}$ or $\ref{case2c}$.
By Table \ref{table2}, the length of the differential of the curvature is non-zero at every point where the metric is defined, in particular at $(\bar{x},\bar{y})$.
By the isometry from the regular part of $\gamma$ into the standard model, and by continuity, the same holds at $\gamma(0)$. Hence, by Lemma \ref{prop.principal} contained in the Appendix, there exists a nowhere vanishing Killing vector field in a neighbourhood $U$ of $\gamma(0)$.
{In particular, $\mathfrak p(g|_U)$ contains a Killing field. Via the isometry with a standard model, it can be identified with a Lie subalgebra of $\mathfrak p(\gst)$. By Remark \ref{rem:Sigma.notempty}, any Lie subalgebra containing a Killing field acts locally transitively. Hence $\mathfrak p(g|_U)$ acts locally transitively at $\gamma(0)$.}
Applying Lemma \ref{lemma.isometry.1}, after possibly restricting to a neighbourhood of $\gamma((0,\varepsilon])$ with $\varepsilon \leq 1$, the isometry of the regular portion of $\gamma$ extends across $\gamma(0)$, contradicting its singularity.
\end{proof}
Thus, suspicious geodesics are the only possible candidates for the realizations, in standard models, of modelled geodesic arcs issuing from singular points; however, this condition is only necessary.
 Indeed, for instance, along such a realization, all  scalar invariants of $g$ must remain bounded.
This provides a restriction that will be used repeatedly in the subsequent analysis.  To this end, we will mainly work with the following scalar invariants:
\begin{enumerate}
\item the scalar curvature $R$
of $g$;
\item the square of the length of the differential of the scalar curvature
    $I=\sum_{i,j} g{}^{ij}
    \frac{\partial R}{\partial x^i}\frac{\partial R}{\partial
    x^j}$;
\item the laplacian of the scalar curvature $\Delta
R=\frac{1}{\sqrt{\det(g{})}} \sum_{i,j} \frac{\partial }{\partial
x^i} \left( g^{ij} \sqrt{\det(g)}\frac{\partial R}{\partial
x^j}\right)$.
\end{enumerate}
For the standard metrics of Theorem \ref{th.Bryant.Manno.Matveev}, we computed the scalar invariants $R$, $I$, and $\Delta R$, and collect their expressions in the tables below for the reader’s convenience.

\begin{table}[ht]
\centering
\small
\renewcommand{\arraystretch}{1.1}
\setlength{\tabcolsep}{4pt}
\begin{tabular}{@{}cccc@{}}
\toprule
& Metric $\ref{case1a}$ & Metric $\ref{case1b}$& Metric $\ref{case1c}$ \\
\midrule
$R$ 
& $\epsilon_{1} b e^{-(b+2)x}$
& $\dfrac{\epsilon_{2} b e^{-(b+2)x}\bigl((b+2)e^{bx}+2\epsilon_{2}\bigr)}{2a}$
& $-\dfrac{e^{-2x}(2x+1)}{2a}$ \\[0.6em]
\midrule
$I$
& $\epsilon_{1} e^{-3(b+2)x} b^{2}(b+2)^{2}$
& $\dfrac{b^{2}(b+2)^{2}e^{-3(b+2)x}(e^{bx}+\epsilon_{2})^{4}}{a^{3}}$
& $\dfrac{4x^{4}e^{-6x}}{a^{3}}$ \\[0.6em]
\midrule
$\Delta R$
& $b(b+2)(b+3)e^{-2(b+2)x}$
& $-\dfrac{\epsilon_{2}b(b+2)(e^{bx}+\epsilon_{2})^{2}e^{-2(b+2)x}
\bigl(be^{bx}-6e^{bx}-6\epsilon_{2}-2b\epsilon_{2}\bigr)}{2a^{2}}$
& $-\dfrac{3x^{2}e^{-4x}(2x-1)}{a^{2}}$ \\
\bottomrule
\end{tabular}
\caption{Scalar invariants $R$, $I$ and $\Delta R$ for the standard metrics $\ref{case1a}$, $\ref{case1b}$ and $\ref{case1c}$.}
\label{table1}
\end{table}

\medskip
\begin{table}[ht]
\centering
\small
\renewcommand{\arraystretch}{1.1}
\setlength{\tabcolsep}{4pt}
\begin{tabular}{@{}cccc@{}}
\toprule
& Metric $\ref{case2a}$ & Metric $\ref{case2b}$& Metric $\ref{case2c}$ \\
\midrule
$R$ 
& $\epsilon_{1}e^{-3x}$
& $\dfrac{\epsilon_{2}(3e^{-2x}+2\epsilon_{2}e^{-3x})}{2a}$
& $\dfrac{8x^{3}+6cx^{2}+12\epsilon_{2}x+\epsilon_{2}c}{2a}$ \\[0.6em]
\midrule
$I$ 
& $9\epsilon_{1}e^{-9x}$
& $9\dfrac{(e^{x}+\epsilon_{2})^{4}e^{-9x}}{a^{3}}$
& $36\dfrac{x(2x^{2}+cx+\epsilon_{2})^{4}}{a^{3}}$ \\[0.6em]
\midrule
$\Delta R$ 
& $12e^{-6x}$
& $3\dfrac{\epsilon_{2}(e^{x}+\epsilon_{2})^{2}e^{-6x}(5e^{x}+8\epsilon_{2})}{2a^{2}}$
& $3\dfrac{(2x^{2}+cx+\epsilon_{2})^{2}(2\epsilon_{2}+16x^{2}+5cx)}{a^{2}}$ \\
\bottomrule
\end{tabular}
\caption{Scalar invariants $R$, $I$ and $\Delta R$ for the standard metrics  $\ref{case2a}$, $\ref{case2b}$ or $\ref{case2c}$.}
\label{table2}
\end{table}
We now record a further restriction on suspicious geodesics which are realizable as image under isometries of modelled geodesics issuing from singular points:  their suspicious behaviour can be investigated by the first coordinate alone.
\begin{Lemma}\label{lem:reduction.x}
Let  $(M,g)$  be a pseudo-Riemannian manifold with $\dim\mathfrak{p}(g)\geq 2$.
Let $\gast(t)=(x(t),y(t))$, be the image under an isometry in a standard model $(\mathcal M,\gst)$ of a modelled geodesic issuing from a singular point of $M$. Then, as $t\to 0^+$, either
$$
x(t)\to \pi_x(\partial\mathcal M)
\qquad\text{or}\qquad
x(t)\to \pm\infty,
$$
where $\pi_x:\mathbb R^2\to\mathbb R$ denotes the projection onto the first coordinate.
\end{Lemma}
\begin{proof}
Assume by contradiction that $x(t)$ neither tends to $\pi_x(\partial\mathcal M)$ nor to $\pm\infty$ as $t\to0^+$. 
Since, for geodesics of standard models, the function $y(t)$ is either constant or strictly monotone (see Lemma \ref{lemma.self.inters}), the limit $\lim_{t\to0^+} y(t)$ always exists.

\smallskip
If $y(t)$ admits a finite limit, then $x(t)$ cannot admit any limit as $t\to0^+$. Indeed, by assumption, $|x(t)|$ does not tend to $+\infty$. Moreover, if it converged to some finite value $\bar x\in\mathbb R$, namely
$$
\gast(t)=(x(t),y(t))\to (\bar x,\bar y)\in \R^2,
$$
being $\gast$  suspicious, such limit point could not belong to $\mathcal M$. Hence necessarily $(\bar x,\bar y)\in\partial\mathcal M$, that is, $\bar x\in\pi_x(\partial\mathcal M)$, contradicting our assumption. Therefore $x(t)$ admits no limit.

Since, for all standard models, the scalar curvature $R$ depends only on $x$ and is monotone on each connected component of $\mathcal M$  (cf. Tables \ref{table1} and \ref{table2}), it follows that $R(\gast(t))$ cannot admit a finite limit as $t\to0^+$. This contradicts the aforementioned necessary condition about scalar invariants.

\smallskip
Assume now that $|y(t)|\to +\infty$ as $t\to0^+$. Hence,  $\dot y(t)$ cannot remain bounded as $t\to 0^+$. By considering the first equation of \eqref{syst.principal}, it follows that $\gst_{22}(x(t))\to 0$ as $t\to 0^+$, namely
$
x(t)\to \pi_x(\partial\mathcal M)$,
a contradiction.

\smallskip
Since in both cases we arrive at a contradiction,  the claim follows.
\end{proof}

\subsection{On the extendability to singular points of metrics of Theorem \ref{th.Bryant.Manno.Matveev}}\label{sec.extendability}

In this section, as a first step, we use the results obtained so far (together with those of the Appendix),
to select which standard metrics are candidates to be glued in a larger domain containing singular points. We will do it by studying metric invariants, introduced in the previous section, along suspicious geodesics.

The first result is proved in Section \ref{sec.1c.2a.2b.not}, where we show that if $(M,g)$ contains regular points of type  \ref{case1c}, \ref{case2a} or \ref{case2b}, then it cannot contain singular ones. 

The remainder of the section is devoted to the remaining cases, to which we apply the same strategy:
we study the behaviour of metric invariants, also along canonically constructed vector fields,
to obtain necessary restrictions on the possible types of regular points and on the parameters of the standard metrics which can occur near singular points. The global conclusion that regular points are all of the same type will be proved in Section \ref{sec.regular.singular}.

\subsubsection{Non-extendability to singular points of metrics of type \ref{case1c}, \ref{case2a}  and \ref{case2b}}\label{sec.1c.2a.2b.not}

\begin{Prop}\label{impossible}
Let $(M,g)$ be a $2$-dimensional pseudo-Riemannian manifold with $\dim\mathfrak{p}(g)\ge 2$. Assume that $M$ contains regular points admitting a neighbourhood isometric to an open subset of one of the standard models \ref{case1c}, \ref{case2a} or \ref{case2b}. Then $M$ has no singular points. Furthermore, all regular points of $M$ are locally isometric to the same standard model.
\end{Prop}
\begin{proof}
Assume by contradiction that $M$ admits singular points. By Corollary \ref{cor:different.types}, there exists a modelled geodesic whose image $\gast(t)=(x(t),y(t))$ lies in one of the standard models $\mathcal M$ of type \ref{case1c}, \ref{case2a} or \ref{case2b} and reaches a singular point.
By Lemma \ref{lem:reduction.x}, as $t\to0^+$, either
$x(t)\to \pi_x(\partial\mathcal M)$
or $x(t)\to \pm\infty$.
We analyse these possibilities case by case.

\smallskip
\noindent\textbf{Case of metrics \ref{case1c}.}

\noindent
Here  $\mathcal M=\mathbb R^2\setminus\{x=0\}$.
In this case $\dim\mathfrak{p}(\gst)=2$, hence $\dim\mathfrak{p}(g)=2$. 
Since the standard model admits a Killing vector field, $M$ carries a projective vector field that is locally Killing. By Proposition \ref{homo}, this field is actually globally Killing, and its length defines a scalar invariant of $g$ that diverges as $x \to 0$.
Therefore, this invariant is unbounded along any geodesic approaching the line $x=0$, contradicting the boundedness of scalar invariants along modelled geodesics.
As $x\to -\infty$, the scalar curvature $R$ diverges (cf. Table \ref{table1}).
Since, by \eqref{syst.principal}, one has
$$
\dot x^{\, 2}=\frac{\left( aC_2-\varepsilon C_1^2x\right) x^{2}}{a^2\, e^{2x}},
$$
the affine parameter is unbounded as $x\to+\infty$, so these geodesics are not suspicious.
Thus, no suspicious geodesic can arise as the image of a modelled geodesic.

\smallskip
\noindent\textbf{Case of metrics \ref{case2a}.}

\noindent
The standard model is defined on the whole $\mathbb{R}^2$.
As $x \to -\infty$, the scalar curvature $R$ diverges (cf. Table \ref{table2}). On the other hand,
it follows from  \eqref{syst.principal}  that
$$
\dot x^{\, 2}=\frac{\varepsilon_1\left( C_2 e^x -\varepsilon_2 C_1^2\right)}{e^{4x}}.
$$
Therefore, the affine parameter is unbounded as $x\to+\infty$.
Thus, no suspicious geodesic can arise as the image of a modelled geodesic.

\smallskip
\noindent\textbf{Case of metrics \ref{case2b}.}

\noindent
If $\epsilon_2=-1$, then $\mathcal M=\mathbb R^2\setminus\{x=0\}$; otherwise $\mathcal M=\mathbb R^2$.
As $x\to -\infty$, the scalar curvature $R$ diverges (cf. Table \ref{table2}). Moreover, as $x\to 0$ or $x\to +\infty$, the affine parameter is unbounded. Indeed, in view of  \eqref{syst.principal}, one has
$$
\dot x^{\,2}= \left(\frac{e^x+\varepsilon_2}{a\, e^{2x}}\right)^2\,\left( a C_2 e^x -\varepsilon_1 C_1^2(e^x+\varepsilon_2)\right).
$$
Thus, no suspicious geodesic can arise as the image of a modelled geodesic.

\medskip
This shows that regular points locally isometric to the standard models \ref{case1c}, \ref{case2a} or \ref{case2b} cannot occur in the presence of singular points.
Moreover, the final statement follows from Remark \ref{lem:several.reg}.
\end{proof}

\subsubsection{On the extendability to singular points of metrics of type \ref{case1a}}\label{sec:1a}


In this section we  analyse scalar invariants along modelled geodesics approaching singular points and then we determine the admissible parameters of regular points of type \ref{case1a} under the assumption that also singular points are present on $M$. 

We preliminarily notice that, since $(M,g)$ contains regular points of type \ref{case1a}, necessarily $\dim \mathfrak{p}(g)=2$ as every standard model of type \ref{case1a} has a $2$-dimensional projective algebra. In particular, $(M,g)$ admits a global Killing vector field $K$, unique up to a multiplicative constant.
Moreover, the projective algebra of the model \ref{case1a} contains also a non-Killing homothety. Hence, in view of Proposition \ref{homo}, $(M,g)$ admits a global non-Killing homothety.

\smallskip
To sum up, \emph{$\mathfrak p(g)$ is generated by a Killing vector field and a non-Killing homothety}.

\medskip
\noindent
In what follows, $\gamma$ will denote a modelled geodesic of $(M,g)$  issuing from a singular point,  whose image $\gast(t)=(x(t),y(t))$ lies in a standard model $(\mathcal M,\gst)$ of type \ref{case1a}.
Since $\mathcal M=\R^2$, the geodesic $\gast$ must be escaping. Also, in view of Lemma \ref{lem:reduction.x}, one necessarily has $\lim_{t\to 0^+}x(t)=\pm\infty$.
In order to understand whether singular points can occur, we first study the behaviour of the invariants listed in Table \ref{table1} along $\gast$. In particular, $R$ diverges in the following cases:
$$
 b<-2 \ \text{and}\  x\to +\infty
\quad\text{or}\quad
 b>-2 \ \text{and}\  x\to -\infty.
$$
In order to further restrict the admissible cases, we now take into account \eqref{syst.principal}. Eliminating $\dot y$ from the second equation, we obtain that, along $\gast$,
$$
\dot x^{\,2}
=
\frac{C_2\,\varepsilon_2 e^{b x}-C_1^{2}}
{\varepsilon_1\varepsilon_2\, e^{2(b+1)x}}.
$$
Combining the requirement that the affine parameter be finite with the previous restrictions, we conclude that the only possible case is
\begin{equation}\label{eq:cond.1a}
-2<b<-1
\qquad \text{and} \qquad
x\to +\infty.
\end{equation}
Moreover, in this range one necessarily has $C_1\neq 0$. Indeed, if $C_1=0$, then the affine parameter is unbounded. In particular, along any geodesic $\sigma$ of $(M,g)$ containing at least one regular point of type  \ref{case1a}, one has $g(\dot\sigma,K)=:C_1\neq 0$. 
Hence, \emph{along any such geodesic, the Killing vector field $K$ is nowhere vanishing}.

\smallskip
When condition \eqref{eq:cond.1a} holds, the length of the Killing vector field $\partial_y$ along $\gast$ tends to zero. If the metric $g$ is Riemannian, this implies that $K$ vanishes at the singular point $\gamma(0)$. This contradicts our previous observation. 
Hence, \emph{the Riemannian case is excluded, and from now on we restrict to Lorentzian metrics}.

\smallskip
Now, in order to further restrict the admissible values of the parameter $b$, we exploit the Lorentzian structure. Being $K(\gamma(0))$ non-zero and isotropic, we can choose a vector field $Z$ defined in a neighbourhood of $\gamma$ such that
\begin{equation}\label{eq:ZZ0.ZK1}
g(Z,Z)=0, \quad g(Z,K)=1.
\end{equation}
Since the scalar curvature $R$ is a smooth function on $M$, it follows that, for every $m\in\mathbb N$, the function
$$
Z^m(R):=\underbrace{Z(Z(\dots Z(R)\dots))}_{m\text{-times}}
$$
is well-defined and smooth. In particular, we can choose $Z$ in such a way that, along the regular part of $\gamma$, it reads
\begin{equation}\label{eq:Z.choice}
Z = e^{-(b+1)x}\,\partial_x - \varepsilon_1 e^{-bx}\, \partial_y.
\end{equation}
A direct computation shows that
\begin{equation}\label{july}
Z^m(R)= \varepsilon_1 b \prod_{h=1}^{m}(hb+h+1)\; e^{-\left((m+1)b+m+2\right)x}.
\end{equation}
Since $-2<b<-1$, the exponent in \eqref{july} is positive for $m$ sufficiently large. Hence, boundedness of $Z^m(R)$ along $\gast$  as $x\to +\infty$ forces
\begin{equation}\label{eq:b.h}
b=-\frac{h+1}{h}, \qquad h\in\mathbb N\setminus\{1\}.
\end{equation}

\subsubsection{On the extendability to singular points of metrics of type \ref{case1b}}\label{sec:1b}

In this section we treat the case in which the regular locus of $(M,g)$ contains points of type \ref{case1b}.
We follow the same general strategy as in Section \ref{sec:1a}.

\smallskip\noindent
Throughout the section, $\gamma$  denotes a modelled geodesic of $(M,g)$  issuing from a singular point, whose image $\gast(t)=(x(t),y(t))$ lies in a standard model $(\mathcal M,\gst)$ of type \ref{case1b}.

\medskip\noindent
We preliminarily notice that, since $(M,g)$ contains regular points of type \ref{case1b}, necessarily $\dim \mathfrak{p}(g)=2$. Indeed, every standard model of type \ref{case1b} has $2$-dimensional projective algebra. In particular, by considering Proposition \ref{homo}, $(M,g)$ admits a globally defined Killing vector field $K$, unique up to a multiplicative constant.

%
%
\medskip
We first observe that, if $\varepsilon_2 = 1$, the standard model $\mathcal M$ is defined on the whole $\R^2$, whereas if $\varepsilon_2 = -1$, one has $\mathcal M = \R^2 \setminus \{x = 0\}$.
 In the latter case, Lemma \ref{lem:reduction.x} allows, a priori, the possibility that $x(t) \to 0$ as $t \to 0^+$. We now exclude this possibility. Indeed, the standard model admits the Killing vector field $\partial_y$, whose squared length is $\gst_{22}(x)$, and $\gst_{22}(x) \to \infty$ as $x \to 0$.  Since the squared length of the globally defined Killing field is a smooth function on $M$, it must remain bounded along a modelled geodesic approaching a singular point.
Therefore, $\gast$ cannot approach the line $\{x=0\}$.
Thus, $\lim_{t \to 0^+} x(t) = \pm \infty$.
We now impose the boundedness of scalar invariants along $\gast$. The case $x(t) \to -\infty$ is excluded, since the scalar curvature $R$ diverges as $x \to -\infty$ (cf. Table \ref{table1}).
From Table \ref{table1}, we also notice that $R$ admits a finite limit as $x \to +\infty$ if and only if $b > -2$. To refine this condition, we consider also the Laplacian $\Delta R$. One obtains that it converges as $x \to +\infty$ for
$$
-2 < b < 0, \quad 0 < b < 1, \quad 1 < b \leq 4, \quad \text{or} \quad b = 6.
$$
We now combine these conditions with \eqref{syst.principal}.
From the second equation of \eqref{syst.principal}, we obtain that, along $\gast$,
$$\dot x^{\,2}
=
\frac{(e^{bx}+\varepsilon_2)^2}{a\,e^{(b+2)x}}
\left(
C_2-\frac{\varepsilon_1 C_1^2\,(e^{bx}+\varepsilon_2)}{a\,e^{bx}}
\right)\,.
$$
Requiring the affine parameter to remain finite as $x \to +\infty$ yields the following alternatives:
\begin{align}
\text{either} \qquad 2 < b \leq 4 \ \text{or} \ b = 6, \qquad aC_2 - \varepsilon_1 C_1^2 > 0,
\label{cond:1b-case1}
\\[0.3ex]
\text{or } \qquad-2 < b < -1, \qquad C_1 \neq 0, \qquad \varepsilon_1 \varepsilon_2 = -1.
\label{cond:1b-case2}
\end{align}
Below we shall analyse these two cases separately.

\bigskip\noindent\textbf{The case corresponding to \eqref{cond:1b-case1}.}
In this situation, \emph{the Killing vector field $K$ does not vanish at the singular point $\gamma(0)$} as its length admits a finite non-zero limit along $\gamma$ as $t\to 0^+$.
To further restrict the admissible values of $b$, we introduce a vector field $Z$ defined near $\gamma$ such that
\begin{equation}\label{eq:Z.1b.I.global}
g(Z,K)=0 \,,\quad |g(Z,Z)|=1.
\end{equation}
Such a vector field is well-defined up to the sign of the component transversal to $K$. In particular, on the regular part of $\gamma$, in the coordinates of the standard model, $Z$ can be written explicitly as
\begin{equation}\label{eq:Z.1b.I}
Z =  \frac{\delta}{\sqrt{|a|}} \, \frac{e^{bx} + \varepsilon_2}{e^{\frac{b+2}{2}x}} \, \partial_x \,, \quad \delta\in\{-1, 1\}.
\end{equation}
We now consider the iterated derivatives $Z^h(R)$ of the scalar curvature $R$ along $Z$.
A direct computation shows that,
for every $h\in\N$, there exists a polynomial $P_h$ such that
\begin{equation}\label{eq:ZhR.1b.I}
Z^h(R)=
\frac{\delta^h\varepsilon_2 b}{2a\,|a|^{h/2}}\,
e^{\frac{h(b-2)-4}{2}x}\,
P_h(e^{-bx}).
\end{equation}
These polynomials $P_h$ are defined recursively by
$$
\begin{cases}
P_0(t)=b+2+2\varepsilon_2 t&\\
P_{h+1}(t)
=
(1+\varepsilon_2 t)
\left(
\frac{h(b-2)-4}{2}\,P_h(t)
-
bt\,P_h'(t)
\right)&\\
\end{cases}
$$
Since, as $x\to +\infty$,
$$
P_h(e^{-bx})\to P_h(0)=
\frac{b+2}{2^h}
\prod_{j=0}^{h-1}\bigl(j(b-2)-4\bigr),
$$
boundedness of $Z^h(R)$ necessarily requires that, for a suitable $h\in\N$,
\begin{equation}\label{eq:b.1b.I}
b =2+ \frac{4}{h}.
\end{equation}

\bigskip\noindent\textbf{The case corresponding to \eqref{cond:1b-case2}.}
Since $x \to +\infty$ as $t\to 0^+$, condition $\varepsilon_1\varepsilon_2=-1$ is equivalent to requiring that the metric be \emph{Lorentzian}.
Moreover, taking into account that
$g\bigl(K(\gamma(0)), K(\gamma(0))\bigr)=0$
and considering that along $\gamma$ one has $g(\dot\gamma, K) =: C_1 \neq 0$,
it follows that \emph{$K(\gamma(0))$ is  non-zero and isotropic}. Therefore, we can choose a vector field $Z$ defined in a neighbourhood of $\gamma$ such that
\begin{equation}\label{eq:Z.1b.II.global}
g(Z,Z)=0\,, \quad g(Z,K)=1.
\end{equation}
There are two possible choices for the vector field $Z$. On the regular part of $\gamma$, in the coordinates of the standard model, they read as
\begin{equation}\label{eq:Z.1b.II}
Z
=
\delta\, \frac{(1-\varepsilon_1 e^{bx})^{\frac32}}{a\,e^{(b+1)x}}\,\partial_x
+
\frac{\varepsilon_1-e^{-bx}}{a}\,\partial_y\,, \quad \delta\in\{-1, 1\}.
\end{equation}
A direct computation shows that, for every $h \in \N$,
\begin{equation}\label{eq:ZhR.1b.II}
Z^h(R)=
\frac{\delta^h\, b}{2a^{h+1}}
e^{-\bigl((b+2)+h(b+1)\bigr)x}
\left(1-\varepsilon_1 e^{bx}\right)^{\frac h2}
Q_h\!\left(\varepsilon_1 e^{bx}\right),
\end{equation}
where the polynomials $Q_h$ are recursively defined by
$$
\begin{cases}
Q_0(t)=2-(b+2)t &\\
Q_{h+1}(t)
=
bt(1-t)Q_h'(t)
+
\Bigl(
-(b+2+h(b+1))(1-t)-\tfrac{bh}{2}\,t
\Bigr)Q_h(t)&\\
\end{cases}
$$
Since, as $x\to +\infty$,
$$
Q_h\!\left(\varepsilon_1 e^{bx}\right)\to Q_h(0)=2(-1)^h\prod_{j=0}^{h-1}\bigl(b+2+j(b+1)\bigr),
$$
boundedness of $Z^h(R)$ necessarily requires that, for a suitable $h\in\N$,
\begin{equation}\label{eq:b.1b.II}
(b+2)+h(b+1)=0
\quad \Longrightarrow \quad
b=-\frac{h+2}{h+1}.
\end{equation}

\subsubsection{On the extendability to singular points of metrics of type \ref{case2c}}\label{sec:2c}
In this section we treat the case in which the regular locus of $(M,g)$ contains points of type \ref{case2c}. 
%
%
%
Below, $\gast(t)=(x(t),y(t))$ denotes the image, in a
standard model $(\mathcal M,\gst)$  of type \ref{case2c}, of a modelled geodesic of
$(M,g)$  issuing from a singular point.

%
%

\smallskip\noindent
By Lemma \ref{lem:reduction.x}, as $t \to 0^+$ one has either $x(t) \to \pm\infty$ or $x(t) \to \pi_x(\partial \mathcal M)$. 
The first possibility is immediately excluded, since, in view of Table \ref{table2}, the scalar curvature $R$ diverges as $x \to \pm\infty$. We now analyse the possible behaviour near the boundary of $\mathcal M$.  
We recall that
$$
\mathcal M = \{ (x,y) \in \mathbb{R}^2 \mid x \neq 0,\; 2x^2 + cx + \varepsilon_2 \neq 0 \}.
$$
 Therefore, $\partial \mathcal M$ consists in the union of the sets $\{x=0\}$ and $\{2x^2 + cx + \varepsilon_2 = 0\}$.
We claim that the latter possibility is excluded.  Let $\tilde x$ be such that $2\tilde x^{\,2} + c\tilde x + \varepsilon_2 = 0$. By considering \eqref{syst.principal}, one obtains that, along $\gast$,
\begin{equation}\label{eq:dot.x.2c}
\dot{x}^2 = (2x^2 + cx + \varepsilon_2)^2\;\frac{aC_2 \,x - C_1^2\varepsilon_1(2x^2 + cx + \varepsilon_2)}{a^2}.
\end{equation}
Hence, if $ x \to \tilde x$ as $t \to 0^+$, then the affine parameter of $\gast$ is unbounded
and such geodesics cannot be suspicious. This is a contradiction.

\smallskip
To sum up, the only admissible behaviour is 
$x(t) \to 0$ as $t \to 0^+$.

\subsection{Regular and singular locus of $(M,g)$ and completion of the proof of Theorems \ref{th.main.3} and  \ref{th.main.2}}\label{sec.regular.singular}

In this section we refine the information obtained in Section \ref{sec.extendability} by studying separately the regular and singular loci in the admissible cases.
We underline that throughout the section we complete the proofs of Theorems \ref{th.main.3} and  \ref{th.main.2}, more precisely at the end of Section \ref{sec:regular.locus.2c}.

\smallskip
As usual $(M,g)$ is a 2-dimensional pseudo-Riemannian manifold with  $\dim\mathfrak{p}(g)\geq 2$.

\subsubsection{The regular and singular locus of $(M,g)$ containing points of type \ref{case1a}}
\label{sec:regular.locus.1a}
%
%
We now analyse the possible occurrence of regular points not of type \ref{case1a} in a manifold $(M,g)$ containing points of type \ref{case1a}. Then such points must be locally isometric to a standard model whose projective algebra contains both a Killing vector field and a non-Killing homothety since, as observed in Section \ref{sec:1a}, $\mathfrak p(g)$ is generated by a Killing vector field and a non-Killing homothety.
In view of Proposition \ref{impossible}, such a model cannot be of type \ref{case1c}, \ref{case2a} or \ref{case2b}. Furthermore, since metrics \ref{case1b} and \ref{case2c} do not admit homothetic vector fields, the only remaining possibilities are another standard model of type \ref{case1a}, possibly characterized by different parameters, or a flat metric.

Assume by contradiction that $(M,g)$ admits either regular points admitting a flat neighbourhood or regular points of type \ref{case1a} with to different values of the parameter $b$. By Lemma \ref{lem:regular_boundary} (see also Remark \ref{lem:several.reg}), we may choose a geodesic $\sigma:[0,1]\to M$ passing through a singular point  $s$ such that $\sigma(0)$ is a regular point of type \ref{case1a}, while $\sigma(1)$ either admits a flat neighbourhood or is a regular point of type \ref{case1a} with a different value of $b$.
By Section \ref{sec:1a}, the Killing vector field $K$ is nowhere vanishing along $\sigma$.
Let
\begin{equation}\label{eq:Ireg}
I_{\mathrm{reg}}:=\{t\in[0,1]\mid \sigma(t)\ \text{is a regular point}\}.
\end{equation}
Then $I_{\mathrm{reg}}$ is an open subset of $[0,1]$. Moreover, its complement $[0,1]\setminus I_{\mathrm{reg}}$ cannot contain any nontrivial interval. Indeed, if $\sigma((a,b))$ consisted only of singular points, then either $K$ would be transverse to $\sigma$ on $(a,b)$, in which case its flow would generate an open subset of $M$ consisting of singular points, or $K$ would be tangent to $\sigma$ on $(a,b)$, so that $\sigma((a,b))$ would be contained in an orbit of $K$, and hence the whole $\sigma$ would consist only of singular points. Therefore, both cases are impossible and  $I_{\mathrm{reg}}$ is dense in $[0,1]$.

For each connected component $I$ of $I_{\mathrm{reg}}$ consisting of points of type \ref{case1a}, let $h_I\in\N\setminus\{1\}$ be the integer determining the corresponding parameter $b$, as in \eqref{eq:b.h}.
Since $\sigma(0)$ is a regular point of type \ref{case1a}, the set
$$
\mathcal H
=
\left\{
h_I\mid
I \text{ is a connected component of } I_{\mathrm{reg}}
\text{ consisting of points of type \ref{case1a}}
\right\}$$
is nonempty. Let
$h_0=\min\mathcal H$ and let $I_0$ be a connected component of $I_{\mathrm{reg}}$ of type
\ref{case1a} whose associated integer is $h_0$.
Up to reparametrizing, one may assume that $\sigma|_{I_0}$ is a modelled geodesic arc.  
Being  the Killing vector field $K$ nowhere vanishing along $\sigma$, the vector field $Z$, introduced in Section \ref{sec:1a}, is defined on a sufficiently small neighbourhood of the whole geodesic $\sigma$.
Moreover, $Z$ can be chosen in such a way that, along $\sigma|_{I_0}$, it reads as in \eqref{eq:Z.choice}. 
Taking into account \eqref{eq:b.h}, we obtain, along $\sigma|_{I_0}$,
\begin{equation}\label{eq:Zm.R1a}
Z^{h_0-1}(R)\equiv \varepsilon_1 b \prod_{k=1}^{h_0-1}(k b + k + 1) \neq 0, \quad Z^m(R)\equiv0 \quad \text{for all } m>h_0-1\,.
\end{equation}
We first observe that $I_{\mathrm{reg}}$ does not contain infinitely many connected components consisting of points \ref{case1a} with the same parameter $b$. Indeed, assume by contradiction the existence of  infinitely many arcs $\gamma$ of $\sigma$ whose endpoints are singular  and whose interior points are regular of  type \ref{case1a} with the same parameter $b$. By Remark \ref{lem:two.singular}, these arcs are modelled. Let $\gast(t)=(x(t),y(t))$ be the image of such an arc in the corresponding standard model. In view of \eqref{eq:cond.1a}, one has $x(t)\to +\infty$ as $t$ tends to both endpoints.

\smallskip
We now distinguish two cases:

\begin{description}
\item[$\gast$ has zero length.]
Then $C_2=0$. Since $C_1\neq0$, \eqref{syst.principal} implies that $\dot x(t)$ never vanishes. Hence $x(t)$ is strictly monotone along $\gast$, which is incompatible with $x(t)\to+\infty$ at both endpoints.

\item[$\gast$ has non-zero length.]
Since $x(t)\to+\infty$ at both endpoints, the function $x(t)$ attains a minimum $x_{\min}$ at some interior point. Let $\phi_I$ be the isometry identifying the arc with the corresponding standard model
and write
$(\phi_I)_*K=\lambda_I\partial_y$ for a suitable constant $\lambda_I\neq0$.
Since $C_1=g(\dot\sigma,K)$ and
$C_2=g(\dot\sigma,\dot\sigma)$ are constant along $\sigma$, at the
minimum point one has, by \eqref{syst.principal},\ 
$
C_2\,\gst_{22}(x_{\min})
=
\frac{C_1^2}{\lambda_I^2}$.
Using this relation in the length integral $\ell$ and the explicit expression
of the scalar curvature $R$, one obtains
$$
\sqrt{|R(x_{\min})|}\,\ell=
2\sqrt{|b|}
\int_0^{+\infty}
\frac{e^{(b+1)s}}{\sqrt{1-e^{bs}}}\,ds
>0.
$$
Notice that the last integral depends only on $b$. Since $R\circ\sigma$ is bounded on $[0,1]$, the lengths of all such arcs are bounded from below by a positive constant independent of the arc. Thus, $\sigma$ cannot
contain infinitely many such arcs.
\end{description}
Taking into account the two cases above, we conclude that there are only finitely many arcs of $\sigma$ whose regular points are of type \ref{case1a} with the same parameter $b$.

\smallskip
Let now consider the continuous function
$$
f:[0,1]\longrightarrow\mathbb R,
\qquad
f(t)=Z^{h_0-1}(R)\,|_{\sigma(t)}.$$
If $I$ is a connected component of $I_{\mathrm{reg}}$ consisting of points of type \ref{case1a} such that its associated integer $h_I>h_0$, then, by \eqref{july}, $f(t)=0$ at every singular endpoint $\sigma|_I$.
The same conclusion holds  for  connected component of $I_{\mathrm{reg}}$ consisting of points admitting a flat neighbourhood.

Let $A$ denote the union of all connected components  of $I_{\mathrm{reg}}$ which are of type \ref{case1a} and satisfy $h_I=h_0$, and set
$B=I_{\mathrm{reg}}\setminus A$.
By the definition of $h_0$, the set $A$ is nonempty. Moreover, in view of our assumptions on $\sigma(1)$, $B$ is nonempty as well. Since $I_{\mathrm{reg}}=A\cup B$ is dense in $[0,1]$, we have
$[0,1]=\overline{A}\cup\overline{B}$.
We may therefore choose
$t_*\in\overline A\cap\overline B$.
Necessarily $t_*\notin I_{\mathrm{reg}}$.
Choose a sequence $t_n\in A$ converging to $t_*$.
Since $A$ has only finitely many connected components, after passing to a subsequence we
may assume that all $t_n$ belong to the same component.
As $f|_I$ is constant and non-zero, continuity gives
$
f(t_*)\neq0$.
 On the other hand, choose $u_n\in B$ with $u_n\to t_*$, and denote by $J_n$ the connected component of $I_{\mathrm{reg}}$ containing $u_n$. Since $t_*\notin J_n$, the closed interval with endpoints $u_n$ and $t_*$ contains a singular endpoint $s_n$ of $J_n$. Choosing $s_n$ to be the endpoint of $J_n$ lying between $u_n$ and $t_*$, we have
$
|s_n-t_*|\leq |u_n-t_*|$,
and hence $s_n\to t_*$.
Each $J_n$ is either a flat component or a component of type \ref{case1a} with $h_{J_n}>h_0$. In either case, the preceding discussion gives
$
f(s_n)=0$.
Therefore, we obtain
$
f(t_*)=0$, a contradiction.

\smallskip
It follows that, \emph{if $(M,g)$ possesses a regular point of type \ref{case1a}, all the other regular points are of type \ref{case1a} with the same value of the parameter $b$}. 
Moreover, the previous argument shows that singular points are isolated along any geodesic transverse to $K$. Since the flow of $K$ preserves the singular locus and $K$ does not vanish there, \emph{the singular points of $M$ are organized into pairwise disjoint isolated curves which are integral curves of the Killing vector field $K$}.

\subsubsection{The regular and singular locus of $(M,g)$ containing points of type \ref{case1b}}
\label{sec:regular.locus.1b}

%
We determine the possible types of regular points that may occur on a manifold $(M,g)$ admitting regular points of type \ref{case1b}. We follow the same strategy as in Section \ref{sec:regular.locus.1a} and indicate only the necessary modifications.

We first observe that, in view of Proposition \ref{impossible}, regular points admitting a neighbourhood isometric to a standard model of type \ref{case1c}, \ref{case2a} or \ref{case2b} are excluded. Moreover, by the analysis in Section \ref{sec:regular.locus.1a}, regular points of type \ref{case1a} cannot occur. 
We claim that regular points of type \ref{case2c} cannot occur either. Let $U\subset M$ be an open set consisting of regular points of type \ref{case1b}. Since
$
\dim\mathfrak p(g|_U)=2
$,
 we have $\dim\mathfrak p(g)=\dim\mathfrak p(g|_U)=2$.
Thus, there exist $K,X\in\mathfrak p(g)$ whose restrictions to $U$ correspond to $\partial_y$ and $\partial_x+y\partial_y$, respectively. By Proposition \ref{homo} and Corollary \ref{cor.proj.field.vanishing}, $K$ is a Killing vector field on $M$ and
$[K,X]=K$.
Suppose that $V\subset M$ consists of regular points of type \ref{case2c}. 
Since the Killing vector field of the corresponding standard model is
unique up to a non-zero constant factor, the restriction of $K$ to $V$
corresponds to $\lambda\partial_y$ for some $\lambda\neq0$. Hence,
$
[\partial_y,Y]=\partial_y$
for some $Y\in\mathfrak p(g|_V)$.
By considering the bases of $\mathfrak p(g|_V)$ given in Remark \ref{rem.proj.v.f.standard}, one immediately sees that this is impossible. Thus, regular points of type \ref{case2c} cannot occur in $M$.
The only remaining possibilities are points admitting a neighbourhood of constant curvature and points locally isometric to a standard model of type \ref{case1b}, possibly with different parameters.

As in Section \ref{sec:regular.locus.1a}, we first establish a finiteness property that will be used below. Let
$\sigma:[0,1]\to M$ be a geodesic. For every fixed admissible value of $b$, the set $I_{\mathrm{reg}}$, see \eqref{eq:Ireg}, contains only finitely many connected components $I$ such that $\sigma(I)$ consists of points of type \ref{case1b} with the same parameter $b$.
Assume otherwise that there are infinitely many pairwise disjoint modelled arcs $\sigma|_I$ with both endpoints singular. After passing to a subsequence, we may assume that the parameters $\varepsilon_1,\varepsilon_2$ are fixed.
 Let
$\gast_I(t)=(x_I(t),y_I(t))$ be the image of such an arc in the corresponding standard model. Then $x_I(t)\to+\infty$ at both endpoints.
Set $C_1=g(\dot\sigma,K)$ and $C_2=g(\dot\sigma,\dot\sigma)$.
The case $C_2=0$ is excluded as in Section \ref{sec:regular.locus.1a}, since then $x_I$ is strictly
monotone. Thus $C_2\neq0$ and $x_I$ attains a minimum at an interior point of $I$. Arguing as in Section \ref{sec:regular.locus.1a}, using the first integrals together with the explicit expressions of the metric and
the scalar curvature, one finds that the lengths of the arcs $\sigma_I$ are bounded from below by a positive constant independent of $I$.
This contradicts the existence of infinitely many pairwise disjoint such arcs in $\sigma([0,1])$ with the same parameter $b$ and proves the claim.

We now exclude points admitting a neighbourhood of constant curvature and, at the same time, show that all regular points of type \ref{case1b} have the same value of the parameter $b$. Assume otherwise and choose a geodesic $\sigma$ as in Section \ref{sec:regular.locus.1a}, namely containing regular points of type \ref{case1b} with a different parameter $b$ or with  a neighbourhood of constant curvature.
For each connected component $I$ of $I_{\mathrm{reg}}$ of type \ref{case1b}, let $h_I$ be the integer determining the corresponding admissible value of $b$, as in \eqref{eq:b.1b.I} or
\eqref{eq:b.1b.II}.

Assume first that components with parameter $b$ reading as \eqref{eq:b.1b.I} occur, and let $h_0$ be the minimum of their corresponding integers $h_I$. Let $A$ be the union of the components whose parameter $b$, reading as \eqref{eq:b.1b.I}, such that $h_I=h_0$. Set $B=I_{\mathrm{reg}}\setminus A$. 
If $B\neq\emptyset$, choose
$t_*\in\overline A\cap\overline B$
as in Section \ref{sec:regular.locus.1a}. 
Since $A$ has finitely many connected components, $t_*$ is a singular
endpoint of one of them. Hence $g(K,K)_{\sigma(t_*)}\neq0$. Continuity implies that
$g(K,K)$ is non-zero along $\sigma$ in a neighbourhood of $t_*$.
On the other hand, $g(K,K)$ vanishes at every singular endpoint of a
component corresponding to \eqref{eq:b.1b.II}. Hence no such endpoints
can accumulate at $t_*$. Using the vector field $Z$ defined by
\eqref{eq:Z.1b.I.global} and setting
$
f(t)=Z^{h_0}(R)|_{\sigma(t)}
$,
equation \eqref{eq:ZhR.1b.I} gives $f(t_*)\neq0$. 
On the other hand,
arguing from the $B$-side as in Section \ref{sec:regular.locus.1a}, one obtains singular endpoints
$s_n\to t_*$. For $n$ sufficiently large, the corresponding component cannot satisfy \eqref{eq:b.1b.II}; hence it either has constant curvature or satisfies \eqref{eq:b.1b.I} with $h_I>h_0$.
In both cases $f(s_n)=0$, and therefore $f(t_*)=0$, a contradiction. Thus $B=\emptyset$.

If no component corresponding to \eqref{eq:b.1b.I} occurs, all
components of type \ref{case1b} satisfy \eqref{eq:b.1b.II}. Using
\eqref{eq:Z.1b.II.global} and \eqref{eq:ZhR.1b.II}, the same argument
shows that they all have the same value of the parameter $b$, and
that no constant curvature component occurs.
Therefore, all regular points of $(M,g)$ are of type \ref{case1b}
with the same value of the parameter $b$.

\smallskip
In conclusion, we proved that \emph{all regular points of $(M,g)$ are of type \ref{case1b} with the same value of the parameter $b$}.
Moreover, the previous argument shows that singular points are isolated along any geodesic transverse to $K$. Since the flow of $K$ preserves the singular locus and $K$ does not vanish there, \emph{the singular points of $M$ are organized into pairwise disjoint isolated curves which are integral curves of the Killing vector field $K$}.

\subsubsection{The regular and singular locus of $(M,g)$ containing points of type \ref{case2c}}
\label{sec:regular.locus.2c}\label{sec:singular.locus.2c}

Let  $(M,g)$ be a manifold admitting regular points of type \ref{case2c}.
By Proposition \ref{impossible}, regular points admitting a neighbourhood isometric to a standard model of type \ref{case1c}, \ref{case2a} or \ref{case2b} are excluded. Moreover, by Sections \ref{sec:regular.locus.1a} and \ref{sec:regular.locus.1b}, regular points of type \ref{case1a} and \ref{case1b} cannot occur. Thus, the only remaining possibilities are regular points of type \ref{case2c} (possibly with different parameters) and points admitting a neighbourhood of constant curvature.

\smallskip
Let $\gamma$ be a modelled geodesic arc issuing from a singular point and let
$\gast(t)=(x(t),y(t))$ be its image in the corresponding standard model of type \ref{case2c}. 

\medskip
\noindent\textbf{Case 1: the metric $g$ is Riemannian or negative definite.}

By the analysis of Section \ref{sec:2c}, one has
$x(t)\to0$ a s $t\to0^+$.
Hence the condition that $g$ be Riemannian or negative definite is
equivalent to $\varepsilon_1\varepsilon_2=1$.

We first show that every singular endpoint of a modelled geodesic is isolated in the singular locus.  Let $s=\gamma(0)$ be a singular point. In view of Table \ref{table2}, $I(s)=0$.
Since $g$ is definite, this implies $(dR)_s=0$. Moreover, a direct computation gives
\begin{equation}\label{eq:morse}
\det(g^{-1}\nabla dR)_s=\frac{9}{a^4}\neq0.
\end{equation}
Thus $s$ is a Morse critical point of $R$. After shrinking a neighbourhood $U(s)$, we may therefore assume that
\begin{equation}\label{eq:drneq0}
dR\neq0\qquad\text{on }U(s)\setminus\{s\}.
\end{equation}
In particular, no regular point of $U(s)\setminus\{s\}$ can admit a neighbourhood of constant curvature. Hence, by the preceding exclusions, every regular point of $U(s)\setminus\{s\}$ is of type \ref{case2c}.

Assume by contradiction that there exists another singular point $s'\in U(s)$. Since the regular locus is dense, we may choose a regular point $r\in U(s)$ such that the geodesic joining $s'$ to $r$ does not pass through $s$. By Proposition \ref{prop:admissible.arc.2}, this geodesic contains a modelled arc $\tilde\gamma$. By the preceding observation, $\tilde\gamma$ is of type \ref{case2c}. Its singular endpoint $\tilde s$ is distinct from $s$ (possibly distinct to $s'$ too), whereas, by the same argument as
above, $(dR)_{\tilde s}=0$. This contradicts \eqref{eq:drneq0}. Thus every singular endpoint of a modelled geodesic is isolated.


\smallskip
In conclusion, \emph{the singular locus of $(M,g)$ consists only of isolated points and its regular locus consists only  points of type \ref{case2c}}.

\medskip\noindent\textbf{Case 2:  the metric $g$ is Lorentzian.}
Since $x(t) \to 0$ as $t\to 0^+$, requiring the metric $g$ to be Lorentzian is equivalent to $\varepsilon_1\varepsilon_2=-1$.
By considering Table \ref{table2}, at the singular endpoint $s'$ of $\gamma$ one has
$
I|_{s'}=0$.
Therefore, either $(dR)_{s'}\neq 0$ and $(dR)_{s'}$ is isotropic, or $s'$ is a Morse singularity of the scalar curvature $R$ as follows from \eqref{eq:morse}. We study these two cases separately.

\smallskip\noindent
{\textbf{Case 2.1: $s'=\gamma(0)$ is a singular point such that $(dR)_{s'}\neq 0$.}}
By Lemma \ref{prop.principal} contained in the Appendix, there exists a neighbourhood $U$ of $s'$ admitting a nowhere vanishing Killing vector field $K$. Shrinking $U$ if necessary, we may assume that $dR\neq0$
throughout $U$. In particular, no regular point of $U$ admits a neighbourhood of constant curvature, and hence \emph{every regular point of $U$ is of type \ref{case2c}}.

Let $\sigma$ be a geodesic of $M$ containing the modelled arc $\gamma$. We claim that $K$ is transverse to $\sigma$. Indeed, since
$
g(K(s'),K(s'))$ is proportional to $=\lim_{x\to 0} \gst_{22}(x)=0$,
the vector $K(s')$ is isotropic. If $K(s')$ were tangent to $\sigma$, then
$g(\dot\sigma,K)=0$ at $s'$, and hence everywhere along $\sigma|_U$ since $\sigma$ is a geodesic.  As $K$ is nowhere vanishing, this implies that $\dot\sigma$ is collinear with $K$ along $\sigma|_U$. Therefore, up to reparametrization, $\sigma|_U$ coincides with an integral curve of $K$. Since the flow of $K$ consists of local isometries, it preserves the singular locus. Hence, $\sigma|_U$ consists only of singular points. This contradicts the construction of the modelled arc $\gamma$. Hence $K$ is transverse to $\sigma|_U$, up to shrinking $U$.
It follows that $\sigma|_U$ cannot contain a nontrivial interval of singular points: otherwise, its images under the local flow of $K$ would generate an open set of singular points, contradicting the density of the regular locus. Consequently, the regular points of $\sigma|_U$ form a dense union of open intervals. Since every regular
point of $U$ is of type \ref{case2c}, the corresponding geodesic arcs with singular endpoints are modelled.

Since $K$ is Killing, one has $dR(K)=0$. As $dR$ is nowhere vanishing on $U$ and $\dim M=2$, it follows that
$\ker dR=\operatorname{span}(K)$.
Since $K$ is transverse to $\sigma|_U$, we therefore have
$dR(\dot\sigma)\neq0$
along $\sigma|_U$. Hence $R\circ\sigma$ is strictly monotone.
We claim that $\sigma|_U$ cannot contain a modelled geodesic arc with
two singular endpoints. Indeed, let $\gamma$ be such an arc and let
$
\gast(t)=(x(t),y(t))
$
be its image in a standard model of type \ref{case2c}, with parameters
$(a,c,\varepsilon,-\varepsilon)$. At both singular endpoints one has
$x(t)\to0$. 
Thus $R$ has the same value at the two endpoints of $\gamma$,
contradicting the strict monotonicity of $R\circ\sigma$.
It follows that $\sigma|_U$ contains at most one singular point. 

\smallskip
To sum up, since the flow of the Killing vector field $K$ preserves the singular locus, it follows that, \emph{away from Morse singularities of $R$, the singular locus of $(M,g)$ is the union of isolated lightlike curves: they are integral curves of a Killing vector field}.

\smallskip
It remains to compare the standard models occurring on the two sides
of the singular curve passing through $s'$. Let their parameters be
$
(a,c,\varepsilon,-\varepsilon)$ and 
$(\bar a,\bar c,\bar\varepsilon,-\bar\varepsilon)$ respectively.
By Table \ref{table2},
$$
R\to-\frac{\varepsilon c}{2a},
\qquad
\Delta R\to -\frac{6\varepsilon}{a^2}
$$
at $s'$, and analogously for the barred parameters. Since $R$ and
$\Delta R$ are smooth on $M$, the corresponding limits coincide.
As $a,\bar a>0$ and $\varepsilon,\bar\varepsilon\in\{-1,1\}$, we
obtain
$\bar\varepsilon=\varepsilon$,
$\bar a=a$ and 
$\bar c=c$.
Thus the \emph{regular regions adjacent to $s'$ are described by the same standard model \ref{case2c}}.

\medskip\noindent
{\textbf{Case 2.2: $s'=\gamma(0)$ is a Morse singularity of the scalar curvature $R$.}} 
By Lemma \ref{prop.killing.morse} contained in the Appendix, there exists a neighbourhood $U$ of $s'$ admitting a Killing vector field $K$ which vanishes only at $s'$. Shrinking $U$ if necessary, we may also assume
that
$dR\neq0$ on $U\setminus\{s'\}$.

Let $\sigma_1,\sigma_2$ be the two lightlike geodesics issuing from $s'$. Along each of them,  $g(\dot\sigma_i,K)$ is constant. Since $K(s')=0$, it follows that, along $\sigma_i$,
$
g(\dot\sigma_i,K)=0\,.
$
As $\dot\sigma_i$ is lightlike and $K$ does not vanish on $U\setminus\{s'\}$, this implies that $K$ is collinear with $\dot\sigma_i$ along $\sigma_i$. In particular, $K$ is isotropic along these curves.

Since every regular point of $U\setminus\{s'\}$ is of type \ref{case2c}  and at such points the Killing vector field is
non-isotropic, $\sigma_1$ and $\sigma_2$ consist entirely of singular points.

Since the flow of $K$ preserves the singular locus and $K$ does not vanish on $U\setminus\{s'\}$, every singular point in $U\setminus\{s'\}$ lies on an integral curve of $K$. Moreover, along the singular locus the vector
field $K$ is isotropic, hence these curves are lightlike. In dimension
two, there are only two distinct isotropic directions. Therefore, at most two such curves can pass through $s'$.
Furthermore, if infinitely many singular curves were present in $U$, one could choose a geodesic transverse to $K$ intersecting infinitely many singular points. Since $dR \neq 0$ on $U \setminus \{s'\}$, this contradicts the fact that singular points are isolated along such geodesics, see Case 2.1. 

Hence, \emph{the singular locus  near a Morse singularity of $R$ consists exactly of two lightlike curves intersecting precisely at such singularity}.

\smallskip
We also observe that the regular locus in $U$ consists only of points of
type \ref{case2c} with the same parameters. The argument is identical to the one given in Case 2.1, and therefore we omit the details.

\smallskip
We can now record that Theorem \ref{th.main.3} has been proved. Indeed, in the presence of singular points, Proposition \ref{impossible} excludes the types \ref{case1c}, \ref{case2a}, and \ref{case2b}. The cases
\ref{case1a} and \ref{case1b} were treated respectively in Sections
\ref{sec:regular.locus.1a} and \ref{sec:regular.locus.1b}, while the remaining case \ref{case2c} has
been handled in the present section. In each case, the existence of regular
points of one type forces all regular points to be of that same type.

We can also record another consequence of the preceding analysis: Theorem \ref{th.main.2} is proved. Around regular points, the existence of a nontrivial Killing vector field follows from the local models of Theorem
\ref{th.Bryant.Manno.Matveev} and Remark \ref{rem.proj.v.f.standard}. Around singular points, Proposition
\ref{impossible} excludes the types \ref{case1c}, \ref{case2a} and \ref{case2b}; the cases \ref{case1a} and \ref{case1b} were treated respectively in Sections \ref{sec:regular.locus.1a} and \ref{sec:regular.locus.1b}, while the remaining
case \ref{case2c} has just been handled above. Hence, $(M,g)$ admits a nontrivial Killing vector field in a neighbourhood of every point of $M$.

\subsection{Completion of  the proof of Theorem \ref{th.main}}\label{sec:final.th.main}

In the present section 
we construct some local coordinates reflecting the geometric properties and restrictions obtained in the previous sections, thus completing the proof of Theorem \ref{th.main}.

\smallskip
In view of Theorem \ref{th.main.3} (proved at the end of Section \ref{sec:regular.locus.2c}), if $(M,g)$ is a $2$-dimensional pseudo-Riemannian manifold with $\dim\mathfrak{p}(g)\geq 2$, regular points of $M$ are all of the same type. In view of this, we denote  by $M^{\ref{case1a}}$, $M^{\ref{case1b}}$ and $M^{\ref{case2c}}$ manifolds whose regular points are, respectively, of type $\ref{case1a},\ref{case1b}$ and $\ref{case2c}$.

\subsubsection{Local forms of $(M^{1a},g)$  near singular points}\label{sec:local.1a}

%
Let $s\in M$ be a singular point and let $\mathcal S$ be the orbit of the Killing vector field $K$ through $s$. 
If $U$ is a sufficiently small neighbourhood of $s$, then $\mathcal{S}$ separates $U$ into two connected components, denoted by $U_+$ and $U_-$, consisting only of regular points.
In view of Sections \ref{sec:1a} and \ref{sec:regular.locus.1a}, each of these components is locally isometric to an open subset of a standard model of type \ref{case1a}, corresponding respectively to the parameter triples
$$
(b,\varepsilon,-\varepsilon)
\qquad \text{and} \qquad
(b,\bar\varepsilon,-\bar\varepsilon).
$$
Fix coordinates $(x,y)$ on $U_-$ and $(\bar x,\bar y)$ on $U_+$ induced by these identifications.
On the whole neighbourhood $U$, a vector field $Z$ satisfying \eqref{eq:ZZ0.ZK1} is well-defined. On each connected component of $U\setminus\mathcal{S}$, conditions \eqref{eq:ZZ0.ZK1} determine $Z$ up to the sign of its component orthogonal to $K$.
Without loss of generality, we may assume that on $U_-$ one has
\begin{equation*}
Z=e^{-(b+1)x}\partial_x-\varepsilon e^{-bx}\partial_y.
\end{equation*}
On the other component $U_+$, as said, there are a priori two possibilities:
\begin{equation}\label{eq:Z.choice.bar}
Z=e^{-(b+1)\bar x}\partial_{\bar x}-\bar\varepsilon e^{-b\bar x}\partial_{\bar y}
\quad\text{or}\quad
Z=-e^{-(b+1)\bar x}\partial_{\bar x}-\bar\varepsilon e^{-b\bar x}\partial_{\bar y}.
\end{equation}

By a direct computation on $U\setminus\mathcal{S}$ and by continuity, one has $[K,Z]=0$ on $U$. Moreover, since $K$ and $Z$ are linearly independent, they define a local coordinate system $(u,v)$ such that
$$
Z=\partial_u\,, \quad K=\partial_v\,.
$$
As $\mathcal{S}$ is an integral curve of $K$, the coordinates can be chosen so that
$$
\mathcal{S}=\{u=0\},\qquad
U_+=\{u>0\}, \qquad U_-=\{u<0\}.
$$

We first determine the change of coordinates $(x,y)\mapsto (u,v)$ on the connected component $U_-$:
$$
\begin{cases}
K(u)=0,\\
Z(u)=1,\\
K(v)=1,\\
Z(v)=0,
\end{cases}
\qquad\Longleftrightarrow\qquad
\begin{cases}
u_y=0,\\
e^{-(b+1)x}u_x-\varepsilon e^{-bx}u_y=1,\\
v_y=1,\\
e^{-(b+1)x}v_x-\varepsilon e^{-bx}v_y=0.
\end{cases}
$$
Solving this system, we obtain
$$
u=\frac{e^{(b+1)x}}{b+1}+c, \qquad v=y+\varepsilon e^x+d,
$$
for suitable constants $c,d\in\mathbb{R}$.
Since $b+1<0$, $e^{(b+1)x}\to 0$ as $x\to+\infty$. Since the singular points correspond to the limit $x\to+\infty$ and lie on $\mathcal{S}=\{u=0\}$, it follows that $c=0$. Thus, on $U_-$,
\begin{equation*}
u=\frac{e^{(b+1)x}}{b+1}, \qquad v=y+\varepsilon e^x+d.
\end{equation*}

If the first expression in \eqref{eq:Z.choice.bar} is chosen, then the same computations as above yield $u(\bar x,\bar y) < 0$, contradicting the assumption $U_+ = \{u>0\}$. Hence, this possibility must be excluded, and only the second option in \eqref{eq:Z.choice.bar} can occur. In this case,  similar computations instead lead, as desired, to a function $u$ which is positive:
\begin{equation*}
u=-\frac{e^{(b+1)\bar x}}{b+1}, \quad v=\bar y-\bar \varepsilon e^{\bar x}+D, \quad D\in \mathbb{R}\,.
\end{equation*}
At this point one can derive a necessary condition on the parameters $\varepsilon$ and $\bar\varepsilon$ by comparing the values of $Z^{h-1}(R)$ on $U_-$ and $U_+$. Indeed, on $U_-$, formula \eqref{eq:Zm.R1a} gives
$$
Z^{h-1}(R)=\varepsilon\, b \prod_{k=1}^{h-1}(kb+k+1).
$$
On the other hand, on $U_+$, considering the second expression in \eqref{eq:Z.choice.bar}, similar computations yield
$$
Z^{h-1}(R)=(-1)^{h-1}\bar\varepsilon\, b \prod_{k=1}^{h-1}(kb+k+1).
$$
Since $Z^{h-1}(R)$ is a smooth function on $U$ and both expressions are constant on $U_\pm$, their values must coincide on $\mathcal S$. Therefore, a necessary condition for gluing $U_+$ and $U_-$ along $\mathcal S$ is 
\begin{equation}\label{eq:gluing.epsilon}
\varepsilon = (-1)^{h-1}\bar\varepsilon.
\end{equation}
Equivalently, $\varepsilon=\bar\varepsilon$  if $h$ is odd and $\varepsilon=-\bar\varepsilon$ if $h$ is even.
We finally express the metric $g$ in the coordinates $(u,v)$ and show that the two local expressions glue smoothly across $\mathcal S$.
Indeed,  straightforward computations yield
\begin{equation*}
g|_{U_-}=2\,du\,dv-\varepsilon\left(-\frac{u}{h}\right)^{h+1}dv^2,\quad\quad g|_{U_+}=2\,du\,dv-\bar\varepsilon\left(\frac{u}{h}\right)^{h+1}dv^2.
\end{equation*}
Hence, when \eqref{eq:gluing.epsilon} holds, the two expressions coincide, and the metric extends smoothly across $\mathcal S$. Up to rescaling of the coordinates $(u,v)$, this yields the normal form
$$
g|_U = 2\,du\,dv + (-\varepsilon)^h \left(\frac{u}{h}\right)^{h+1} dv^2.
$$
In a suitable system of null coordinates, the above metrics assume the form of metrics \ref{case1a.main} of Theorem \ref{th.main}, provided the renaming $\varepsilon\leftrightarrow -\varepsilon$.

\subsubsection{Local forms of $(M^{\ref{case1b}},g)$ near singular points}\label{sec:local.1b}

%

To describe the local form of the metric near a singular point, we follow the same strategy as in Section \ref{sec:local.1a}.

\smallskip\noindent
Let $s \in M$ be a singular point and $\mathcal S$  be the orbit of the Killing vector field $K$ through $s$. For a sufficiently small neighbourhood $U$ of $s$, the set $\mathcal S$  separates $U$ into two connected components $U_+$ and $U_-$, consisting only of regular points.
By the results of Sections \ref{sec:1b} and \ref{sec:regular.locus.1b}, each of these components is locally isometric to an open subset of a standard model of type $1b$ with the same parameter $b$. We fix the corresponding coordinate systems $(x,y)$ on $U_-$ and $(\bar{x},\bar{y})$ on $U_+$ so that $K = \partial_y = \partial_{\bar{y}}$.

We now consider the vector field $Z$ introduced in Section \ref{sec:1b}. Depending on the sign of the parameter $b$, this vector field is defined by \eqref{eq:Z.1b.I.global} or \eqref{eq:Z.1b.II.global}. In both cases one checks that $[K,Z]=0$. Therefore, the pair $(K,Z)$ defines a system of local coordinates $(u,v)$ on $U$, characterized by
$$
Z = \partial_u, \quad K = \partial_v, 
$$
and such that
$$
\mathcal S = \{u = 0\}, \qquad U_\pm = \{\pm u > 0\}.
$$
In these coordinates, the problem reduces to comparing the expressions of the metric on $U_+$ and $U_-$, and to determining the compatibility conditions under which they glue smoothly across $S$.

\smallskip
We now treat separately the cases corresponding to the two possible signs of the parameter $b$, cf. \eqref{cond:1b-case1} and \eqref{cond:1b-case2}.

\bigskip\noindent\textbf{Case 1: the parameter $b=2+\tfrac{4}{h}$.}
In this case, we use the vector field $Z$ introduced in Section \ref{sec:1b}, whose expression in the coordinates $(x,y)$ is given by \eqref{eq:Z.1b.I}.
Following  the scheme of Section \ref{sec:local.1a},  we first determine the change of coordinates $(x,y)\mapsto (u,v)$  on $U_-$:
\begin{equation}\label{eq:u.v.minus.1b}
u=-\sqrt{|a|}\int_x^{+\infty}
\frac{e^{(2+\frac{2}{h})\xi}}{e^{(2+\frac{4}{h})\xi}+\varepsilon_2}\,d\xi,
\quad
v=y+d, \,\,\,d\in\mathbb{R}
\end{equation}
and the change of coordinates $(\bar x,\bar y)\mapsto (u,v)$ on $U_+$:
\begin{equation}\label{eq:u.v.plus.1b}
u=\sqrt{|\bar a|}\int_{\bar x}^{+\infty}
\frac{e^{(2+\frac{2}{h})\xi}}{e^{(2+\frac{4}{h})\xi}+\bar\varepsilon_2}\,d\xi,
\quad
v=\bar y+D, \,\,\,D\in\mathbb{R}\,.
\end{equation}
The above change of coordinates has been attained by choosing $\delta$ in \eqref{eq:Z.1b.I} such that 
$u<0$ on $U_-$ and $u>0$ on $U_+$.

\medskip
The compatibility conditions on parameters of $U_\pm$
are obtained by comparing scalar invariants across $\mathcal S$. The continuity of $g(K,K)$ implies $a\varepsilon_1=\bar a\,\bar\varepsilon_1$,
while the comparison of $Z^h(R)$ along $\mathcal S$ (cf. \eqref{eq:ZhR.1b.I}) yields
\begin{equation}\label{eq:gluing.epsilon2}
\bar \varepsilon_2=(-1)^h\, \frac{a}{\bar a}\,\varepsilon_2.
\end{equation}
In order to make explicit the gluing of $g|_{U_+}$ and $g|_{U_-}$ along $\mathcal S$, it is convenient to perform a change of variable in the integrals defining the coordinate $u$. Namely, we set
$$
\tau=-e^{-\frac{2}{h}\xi}\quad\text{in }\eqref{eq:u.v.minus.1b}
\qquad\text{and} \qquad
\tau=e^{-\frac{2}{h}\xi}\quad\text{in }\eqref{eq:u.v.plus.1b},
$$
and we denote by $t$ the endpoint of integration, respectively
$$
t=-e^{-\frac{2}{h}x}\quad\text{on }U_-
\qquad\text{and} \qquad
 t=e^{-\frac{2}{h}\bar x}\quad\text{on }U_+.
$$
This yields
$$
u=-\frac{h\sqrt{|a|}}{2}\int_{t}^{0} \frac{d\tau}{1+\varepsilon_2(-\tau)^{h+2}}
\quad\text{when } t<0,
\qquad
u=-\frac{h\sqrt{|\bar a|}}{2} \int^0_{ t} \frac{d\tau}{1+\bar \varepsilon_2 \tau^{h+2}}
\quad\text{when } t>0.
$$
Thus, $u(t)$ extends smoothly across $t=0$ if and only if $\bar \varepsilon_2=(-1)^h\,\varepsilon_2,$
(and hence, considering \eqref{eq:gluing.epsilon2}, if and only if $a=\bar a $).
Moreover, since
$$
\frac{du}{dt}\bigg|_{t=0}=\frac{h\sqrt{|a|}}{2}\neq 0,
$$
the function $t\mapsto u$ is a local diffeomorphism, so that $(t,v)$ defines a local coordinate system on $U$.
In the coordinates $(t,v)$, the metric $g|_U$ takes the form
$$
g|_U=
\frac{ah^2}{4}\frac{dt^2}{(1+\bar\varepsilon_2 t^{h+2})^2}
+
\frac{a\varepsilon_1}{1+\bar\varepsilon_2 t^{h+2}}\,dv^2.
$$
After the rescaling
$$
\tilde t=\frac{h\sqrt{|a|}}{2}\,t,
\qquad
\tilde v=\sqrt{|a|}\,v,
$$
we obtain
$$
g|_U=
\eta_1\,
\frac{d\tilde t^{\,2}}{(1+\kappa \tilde t^{h+2})^2}
+
\eta_2\,
\frac{d\tilde v^2}{1+\kappa \tilde t^{h+2}}, \quad \eta_1=\operatorname{sign}(a),\,\,
\eta_2=\operatorname{sign}(a)\,\varepsilon_1,\,\,
\kappa=\varepsilon_2\left(\frac{2}{h\sqrt{|a|}}\right)^{h+2}.
$$
Hence, these metrics are parametrised by
$(\eta_1,\eta_2,\kappa)$, $\eta_1,\eta_2\in\{-1,1\}$, $\kappa>0$ if $h$ is odd (in this case metrics with opposite $\kappa$ are isometric) and $\kappa\in\R\setminus\{0\}$ if $h$ is even. Two such metrics are  isometric only if the corresponding triples coincide. 

Up to  renaming some parameters and scaling the coordinates, we arrive at normal forms \ref{case1b1.main}-\ref{case1b1bis.main} of Theorem \ref{th.main}.

\bigskip\noindent\textbf{Case 2: the parameter $b=-1-\tfrac{1}{h+1}$.}
In this case, we consider the vector field $Z$ introduced in Section~\ref{sec:1b}, whose expression in the coordinates $(x,y)$ is given by \eqref{eq:Z.1b.II}. Moreover, we recall that, by \eqref{cond:1b-case2}, the metric is Lorentzian, hence $\varepsilon_2 = -\varepsilon_1$ and $\bar\varepsilon_2 = -\bar\varepsilon_1$

Arguing as in the previous case, we  determine the change of coordinates $(x,y)\mapsto (u,v)$ on $U_-$:
\begin{equation}\label{eq:u.v.minus.1b.neg}
u=-|a|\int_{x}^{+\infty}
\frac{e^{-\frac{\xi}{h+1}}}{\left(1-\varepsilon_1 e^{-\frac{h+2}{h+1}\xi}\right)^{3/2}}\,d\xi,
\quad
v=y-\operatorname{sign}(a)\int_{x}^{+\infty}
\frac{e^{\xi}}{\sqrt{1-\varepsilon_1 e^{-\frac{h+2}{h+1}\xi}}}\,d\xi+d,\,\,\, d\in\mathbb{R}
\end{equation}
and the change of coordinates $(\bar x,\bar y)\mapsto (u,v)$ on $U_+$:
\begin{equation}\label{eq:u.v.plus.1b.neg}
u=|\bar a|\int_{\bar x}^{+\infty}
\frac{e^{-\frac{\xi}{h+1}}}{\left(1-\bar\varepsilon_1 e^{-\frac{h+2}{h+1}\xi}\right)^{3/2}}\,d\xi,
\quad
v=\bar y+\operatorname{sign}(\bar a)\int_{\bar x}^{+\infty}
\frac{e^{\xi}}{\sqrt{1-\bar\varepsilon_1 e^{-\frac{h+2}{h+1}\xi}}}\,d\xi+D, \,\,\, D\in\mathbb{R}\,.
\end{equation}
We now derive a necessary compatibility condition on the parameters of $U_\pm$ by comparing the values of $Z^h(R)$ along $\mathcal S$. Indeed, in view of \eqref{eq:ZhR.1b.II},
$$
Z^h(R)=
\operatorname{sign}(a)^h\,\frac{b}{2a^{h+1}}\,Q_h(0)
\quad\text{on }U_-\,,
\quad
Z^h(R)=
(-1)^h\operatorname{sign}(\bar a)^h\,\frac{b}{2\bar a^{h+1}}\,Q_h(0)
\quad\text{on }U_+.
$$
Since $Z^h(R)$ is a smooth function on $U$, these values must coincide. As $Q_h(0)\neq 0$, it follows that
\begin{equation}\label{eq:gluing.a.1b.neg}
\bar a=(-1)^h\,a.
\end{equation}
In order to make explicit the gluing of $g|_{U_+}$ and $g|_{U_-}$ along $\mathcal S$, it is convenient to perform a change of variable in the integrals defining the coordinate $u$. Namely, we set
$$
\tau=-e^{-\frac{\xi}{h+1}}\quad\text{in }\eqref{eq:u.v.minus.1b.neg},
\quad
\tau=e^{-\frac{\xi}{h+1}}\quad\text{in }\eqref{eq:u.v.plus.1b.neg},
$$
and we denote by $t$ the endpoint of integration, respectively
$$
t=-e^{-\frac{x}{h+1}}\quad\text{on }U_-,
\qquad
 t=e^{-\frac{\bar x}{h+1}}\quad\text{on }U_+.
$$
This yields
$$
u(t)=-(h+1)|a|\int_t^0 \frac{d\tau}{\bigl(1-\varepsilon_1(-\tau)^{h+2}\bigr)^{3/2}}
\quad\text{if } t<0,
\quad
u(t)=-(h+1)|\bar a|\int^0_{ t} \frac{d\tau}{\bigl(1-\bar\varepsilon_1\tau^{h+2}\bigr)^{3/2}}
\quad\text{if } t>0.
$$
The two branches glue to a $C^\infty$ function at $t=0$ if and only if 
\begin{equation}\label{eq:gluing.1b.neg.final}
|a|=|\bar a|\qquad \text{and}\qquad \bar\varepsilon_1 = (-1)^{h}\,\varepsilon_1.
\end{equation}
Moreover, since
$$
\frac{du}{dt}\bigg|_{t=0}=(h+1)|a|\neq 0,
$$
the function $t\mapsto u$ is a local diffeomorphism. Therefore, $(t,v)$ is also a local coordinate system on $U$.
Taking into account \eqref{eq:gluing.a.1b.neg} and \eqref{eq:gluing.1b.neg.final}, in the coordinates $(t,v)$, the metric $g|_U$ takes the form
$$
g|_U=
\frac{2(h+1)|\bar a|}{\bigl(1-\bar\varepsilon_1 t^{h+2}\bigr)^{3/2}}\,dt\,dv
-
\frac{\bar a\,t^{h+2}}{1-\bar\varepsilon_1 t^{h+2}}\,dv^2.
$$
After the rescaling
$$
\tilde t=\frac{t}{\left((h+1)^2|\bar a|\right)^{1/h}},
\qquad
\tilde v=(h+1)^{1+\tfrac2h}\,|\bar a|^{1+\tfrac1h}\, v,
$$
the metric becomes
$$
g|_U=
\frac{2}{\bigl(1-\kappa \tilde t^{\,h+2}\bigr)^{3/2}}\,d\tilde t\,d\tilde v
+
\eta\,
\frac{\tilde t^{\,h+2}}{1-\kappa \tilde t^{\,h+2}}\,d\tilde v^{\,2},\quad \eta=-\operatorname{sign}(\bar a),
\,\,\,
\kappa=\bar\varepsilon_1\bigl((h+1)^2|\bar a|\bigr)^{\frac{h+2}{h}}.
$$
Therefore,  these metrics are parametrised by
$(\eta,\kappa)$, $\eta\in\{-1,1\}$, $\kappa>0$ if $h$ is odd (in this case metrics with opposite $\kappa$ are isometric) and $\kappa\in\mathbb R\setminus\{0\}$ if $h$ is even. Two such metrics are  isometric only if the corresponding couples coincide.

Up to  renaming some parameters and scaling the coordinates, we arrive at normal forms \ref{case1b2.main}-\ref{case1b2bis.main} of Theorem \ref{th.main}.

\subsubsection{Local forms of $(M^{\ref{case2c}},g)$ near singular points}\label{sec:local.2c}

\smallskip\noindent\textbf{Case 1: the metric $g$ is Riemannian or negative definite.}\label{sec.Riemannian.2c.sing}
By Section \ref{sec:regular.locus.2c}, all regular points are of type \ref{case2c} and  singular points are isolated. Hence the regular locus of $(M,g)$ is connected.
Taking into account that the metric is assumed to be Riemannian or negative definite (hence $\varepsilon_1=\varepsilon_2=:\varepsilon$) and in view of Lemma \ref{lem:regular_boundary}, the regular locus is locally isometric to the standard model of type \ref{case2c} associated  to a fixed  parameter quadruple
$
(a,c,\varepsilon,\varepsilon).
$

Let $s \in M$ be a singular point and $U$ be a neighbourhood of $s$ such that $V:=U \setminus \{s\}$ is locally isometric to an open subset of a standard model of type \ref{case2c}. We denote by $(x,y)$ the coordinates on $V$ corresponding to such identification and we set $\eta := \mathrm{sign}(x)$. Then
$$g|_V=a\eta \left(\frac{dx^2}{f(x)^2\, |x|}
          +  \frac{|x| }{f(x) }\, dy^2\right),
\quad 
f(x) = 2\varepsilon\, x^2+c\varepsilon \,x+1.
$$
In order to test whether the metric extends smoothly across $x=0$, we introduce a coordinate system adapted to the degeneracy of the metric. Namely, we define
$$
\rho =\int_0^x\frac{1}{\sqrt{|\xi|}} \, d\xi=2\eta\sqrt{|x|}, 
\quad 
\theta = \frac{y}{2}.
$$
In these coordinates, a direct computation yields
$$
g|_{V}
=
a \eta\left(
\frac{d\rho^2}{f(\eta\rho^2/4)^{\,2}}
+
\frac{\rho^2}{f(\eta \rho^2/4)}\,d\theta^2
\right).
$$

Since $\theta = \frac{y}{2}$, the Killing vector field $\partial_y$ is proportional to $\partial_\theta$ in the new coordinates. Hence, its flow preserves $\rho$ and the curves $\{\rho=\mathrm{const}\}$ coincide with its orbits. In particular, $(\rho,\theta)$ can be regarded as polar-type coordinates adapted to the Killing field, where $\rho$ parametrises the space of orbits.

We now pass to Cartesian coordinates
$u=\rho\cos\theta$, $v=\rho\sin\theta$. A direct computation yields
$$
g|_{V}
=
\frac{a\eta}{ f\!\left(\eta\frac{u^2+v^2}{4}\right)} \left(du^2+dv^2 \right)
+\frac{a}{4}\,
h\!\left(\eta\frac{u^2+v^2}{4}\right)
\left(\frac{u\,du+v\,dv}{ f\!\left(\eta\frac{u^2+v^2}{4}\right)}\right)^2, \quad h(s) = -2\varepsilon\, s-c\varepsilon.
$$
Observe that all coefficients extend smoothly at the singular point $s$ (corresponding to $u=v=0$).

\smallskip
Two such metrics are isometric if and only if the corresponding quadruples $(a,c,\eta,\varepsilon)$ coincide. 
Indeed, $\eta=\pm1$ is determined by the signature of $g$. 
Moreover, in view of Table \ref{table2}, the scalar invariants $R$ and $\Delta R$ uniquely determine the parameters $a$, $c$ and $\varepsilon$. 

\noindent
Up to obvious renaming, the above metrics coincide with metrics \ref{case2c.singular} of Theorem \ref{th.main}.

\bigskip\noindent\textbf{Case 2:  the metric  $g$ is Lorentzian}

\medskip\noindent\textbf{Case 2.1:  the Lorentzian  metric  $g$  near a Morse singularity of the scalar curvature $R$. }
Let $s\in M$ be a Morse singularity of the scalar curvature $R$. By Lemma \ref{prop.killing.morse} contained in the Appendix,  there exists a Killing vector field $K$ on a neighbourhood $U$ of $s$, which vanishes only at $s$. 
Moreover, by Section \ref{sec:singular.locus.2c}  Case 2.2, the singular locus in $U$ consists of two lightlike curves $\mathcal S_1$ and $\mathcal S_2$, both integral curves of $K$, intersecting precisely at $s$.
The complement
$U\setminus (\mathcal S_1\cup \mathcal S_2)$
has four connected components, all consisting of regular points of type \ref{case2c}. As observed at the
end of Section \ref{sec:singular.locus.2c}, the parameters of the corresponding standard models are the same on all these components.

Let $\Omega$ be one of the connected components of
$U\setminus(\mathcal S_1\cup\mathcal S_2)$. On $\Omega$, we may write
$$g|_\Omega=a\left(\frac{dx^2}{f^2(x)}-\frac{x}{f(x)}\,dy^2\right),\quad f(x)=1+c\varepsilon x+2\varepsilon x^2.
$$
In order to test whether the metric extends smoothly across the singular locus, we introduce a coordinate system adapted to the degeneracy of the metric. Namely, we define
$$
\rho =\int_0^x\frac{1}{\sqrt{|\xi|}} \, d\xi=2\eta\sqrt{|x|}, 
\quad 
t = \frac{y}{2}, \quad\eta=\operatorname{sign}(x)\,.
$$
In these coordinates, a direct computation yields
$$g|_\Omega=a\eta\left(\frac{d\rho^2}{f(\eta\rho^2/4)^2}-\frac{\rho^2}{f(\eta\rho^2/4)}\,dt^2\right).
$$
Since $t = \frac{y}{2}$, the Killing vector field $\partial_y$ is proportional to $\partial_t$ in the new coordinates. Hence, its flow preserves $\rho$ and the curves $\{\rho=\mathrm{const}\}$ coincide with its orbits. In particular, $(\rho,t)$ can be regarded as the Lorentzian analogue of polar-type coordinates obtained in the Riemannian case above. By introducing the hyperbolic Cartesian coordinates
$u=\rho\cosh t$, $v=\rho\sinh t$
and by putting
$$
s=\eta\frac{u^2-v^2}{4},
$$
we obtain
$$g|_\Omega=a\left(\eta\frac{du^2-dv^2}{f(s)}+\frac{h(s)}{f(s)^2}(u\,du-v\,dv)^2\right),
\quad h(s)=-\frac{c\varepsilon}{4}-\frac{\varepsilon}{2}s\,.
$$
Observe that all coefficients extend smoothly  across $s=0$. Hence the metric extends smoothly across the singular locus (corresponding to $u^2-v^2=0$).

\smallskip
Two such metrics are isometric if and only if the corresponding triples $(a,c,\varepsilon)$ coincide. The sign $\eta$ is not essential, since the local isometry exchanging the coordinates $u$ and $v$ identifies the two choices $\eta=\pm1$.
Up to this identification, and after the obvious renaming and rescaling of variables, the above metrics coincide with metrics \ref{case2c.singular.bis} of Theorem \ref{th.main}.

\medskip\noindent\textbf{Case 2.2:  the Lorentzian metric  $g$  near singular points where  $dR\neq 0$.}
Let $s\in M$ be a singular point such that $dR_s \neq 0$.  By Lemma \ref{prop.principal} contained in the Appendix, there exists a nowhere vanishing Killing vector field $K$ on a neighbourhood $U$ of $s$. 
By Section \ref{sec:singular.locus.2c} Case 2.1, after possibly shrinking $U$, the singular locus in $U$ is an isolated lightlike curve $\mathcal S$, which is an integral curve of $K$. As a consequence,  the set $\mathcal{S}$ separates $U$ into two connected components, denoted by $U_+$ and $U_-$, both consisting only of regular points of type \ref{case2c}, with the same parameters $(a,c,\varepsilon,-\varepsilon)$.

We now introduce a vector field $Z$ on $U$ by imposing
$$
g(Z,Z)=0, \qquad g(K,Z)=\left(-\frac{g(K,K)}{I}\right)^{\frac{3}{10}},
$$
where $I:=g^{-1}(dR,dR)$.  The particular normalization is chosen so as to
simplify computations in the coordinate
construction below and the resulting expression of the metric. We first verify that it is legitimate:
we show that the function $\left(-\frac{g(K,K)}{I}\right)^{\frac{3}{10}}$ is smooth
and well-defined on the whole neighbourhood $U$, including all singular points.
Since $K$ is Killing, $dR(K)=0$, i.e., $dR\in \operatorname{Ann}(K)$.
As $\dim M=2$ and $K\neq 0$, the annihilator $\operatorname{Ann}(K)$ is
$1$-dimensional and is generated by $\iota_K\operatorname{vol}_g$. Therefore,
there exists a smooth function $\varphi$ on $U$ such that
$
dR=\varphi\,\iota_K\operatorname{vol}_g.
$
Since $dR\neq 0$ on $U$, it follows that $\varphi\neq 0$ everywhere.
On a 2-dimensional Lorentzian manifold one has
$$
g^{-1}(\iota_K\operatorname{vol}_g,\iota_K\operatorname{vol}_g)=-g(K,K),
$$
and hence
$$
I=g^{-1}(dR,dR)=-\varphi^2\,g(K,K).
$$
Therefore,
$$
\left(-\frac{g(K,K)}{I}\right)^{\frac{3}{10}}=\varphi^{-\frac{3}{5}}
$$
is a smooth nowhere vanishing function on $U$. 
On each connected component $U_\pm$, using the coordinates of the standard model
of type \ref{case2c} and considering that in such coordinates
$$g|_{U_\pm}=a \left(\frac{dx^2}{f(x)^2\, x}
          -  \frac{x }{f(x) }\, dy^2\right),
\quad 
f(x) = 2\varepsilon\, x^2+c\varepsilon \,x+1,$$
 the vector field $Z$ is therefore of the form
$$
Z=
\left(\frac{a^2}{6}\right)^{\frac35}
\left(
\frac{\delta}{a}\,\partial_x
-
\frac{1}{a x\sqrt{f(x)}}\,\partial_y
\right),
\quad \delta\in\{-1,1\}.
$$
Moreover, since in these coordinates $K=\partial_y$, it
is immediate to check that $[Z,K]=0$ on the regular locus $U\setminus\mathcal S$, and therefore, by density, on the whole of $U$. Hence, $K$ and $Z$ define local coordinates $(u,v)$ on $U$ such that
$$
Z=\partial_u\,,\quad K=\partial_v\,.
$$
Moreover, since $\mathcal S$ is an integral curve of $K$, we may choose these
coordinates in such a way that
$$
\mathcal S=\{u=0\},\qquad U_\pm=\{\pm u>0\}.
$$
Since
$$
K(u)=0,\quad K(v)=1,\qquad Z(u)=1,\quad Z(v)=0,
$$
a direct computation yields
$$
u=
\delta\,a\left(\frac{a^2}{6}\right)^{-\frac35}x,
\quad
v=
y+\delta\int^x\frac{ds}{s\sqrt{f(s)}}.
$$
In these coordinates, the metric takes the form
$$
g =
2\left(\frac{a^2}{6}\right)^{\frac35}
f(\kappa u)^{-\frac32}\,du\,dv
-
\frac{a\,\kappa u}{f(\kappa u)}\,dv^2,
\quad
\kappa = \delta a^{-1}\left(\frac{a^2}{6}\right)^{\frac35}.
$$
We observe that the right-hand side defines a smooth tensor field on the whole neighbourhood $U$, since $f(0)=1$ and therefore the metric extends smoothly across $\mathcal S=\{u=0\}$.

Up to a rescaling of the coordinates $(u,v)$ and a renaming of the parameters,
this yields the normal form
$$\kappa
\left(
\frac{2}{f^{\frac32}(x)}\,dx\,dy
-
\frac{x}{f(x)}\,dy^2
\right),\quad
f(x)=1+c\varepsilon x+2\varepsilon x^2 \,,\,\,\kappa>0.
$$

We now show that this normal form is not new, since it is nothing but the
normal form \ref{case2c.singular.bis}  written in a neighbourhood of a point of the singular locus
different from the intersection point of its two branches.
 Since $f(0)=1$, the function
$$
\psi(x):=\int_0^x\frac{1-f(s)^{-\frac12}}{2s}\,ds
$$
is smooth near $x=0$. Define
$$
\tilde x=\frac{e^{\frac y2+\psi(x)}+xe^{-\frac y2-\psi(x)}}{2},\quad
\tilde y=\frac{e^{\frac y2+\psi(x)}-xe^{-\frac y2-\psi(x)}}{2}.
$$
This gives a local change of coordinates near every point of $\{x=0\}$.
Putting
$$
t=\tilde x^2-\tilde y^2,\quad h(t)=-2\varepsilon t-c\varepsilon ,
$$
straightforward, but lengthy, computations give
$$
4\kappa\left(
\frac{d\tilde x^2-d\tilde y^2}{f(t)}
+
\frac{h(t)}{f(t)^2}
\left(\tilde x\,d\tilde x-\tilde y\,d\tilde y\right)^2
\right)
=
\kappa\left(
\frac{2}{f(x)^{\frac32}}\,dx\,dy
-
\frac{x}{f(x)}\,dy^2
\right).
$$
After renaming $4\kappa$ as $\kappa$, this is precisely the normal form
\ref{case2c.singular.bis}.

\section*{Appendix: Killing vector fields as parallel sections of the Kostant connection and their extension}


In this section we give some extension results concerning Killing vector fields on a $2$-dimensional pseudo-Riemannian manifold $(M,g)$, in particular Lemma \ref{prop.principal} and Lemma \ref{prop.killing.morse}.

\smallskip
In order to do it, as a first step, we briefly recall the known fact that Killing vector fields of a pseudo-Riemannian manifold $(M,g)$ are in one-to-one correspondence with $\widetilde\nabla$-parallel sections of the \emph{Kostant bundle} 
$$
TM \oplus \Lambda^{1}_{1}(M),
$$
where $\Lambda^{1}_{1}(M)$ is the bundle of $g$-skew-symmetric endomorphisms of $TM$ and the \emph{Kostant connection} $\widetilde\nabla$ is given by
\begin{equation}\label{kostant}
\widetilde\nabla_X(Y,A)
:=\bigl(\nabla_X Y - A(X),\;\nabla_X A + R(Y,X)\bigr),
\end{equation}
with $\nabla$ the Levi-Civita connection and $R$ the curvature tensor (one can consult also \cite{console,discala} for more details). 
Indeed, if $Y$ is a Killing vector field and one sets $A:=\nabla Y$, then the
Killing equation implies that $A$ is $g$-skew-symmetric.
Moreover, the  second covariant derivative identity for Killing fields
yields $\nabla_X A + R(Y,X)=0$, hence $\widetilde\nabla(Y,A)=0$.
Conversely, if $(Y,A)$ is a parallel section of the Kostant bundle, then
 $\nabla_X Y=A(X)$ for all $X$, and since $A$ is $g$-skew-symmetric, $Y$ satisfies the Killing equation.

\smallskip
We now concentrate on the case of $2$-dimensional pseudo-Riemannian manifolds $(M,g)$.
Fix a local coordinate system $(x^1,x^2)$ on an open set $U$ of 
$M$, and let $(Y,A)$ be a local section of the Kostant bundle over $U$. Writing
$$
Y=Y^i\partial_i,\qquad A=A^i_j\,\partial_i\otimes dx^j,
$$
we compute the components of $\widetilde\nabla(Y,A)$:
\begin{equation}\label{kostantloc}
\widetilde\nabla_{\partial_k}(Y,A)
\overset{\eqref{kostant}}{=}\Bigl(\left(Y_{,k}^i-A_k^i\right)\partial_i\,,\, \left(A^i_{j,k}+R^i_{j\ell k}\,Y^\ell\right)\partial_i  \otimes dx^j\Bigr),
\end{equation}
where commas denote covariant derivatives w.r.t. the Levi-Civita connection of $g$. Let
$
\epsilon := \sqrt{|\det(g)|}\,dx^1\wedge dx^2
$
denote the pseudo-Riemannian volume form on $U$, and let $\epsilon_{ij}$ be its
components in the chosen coordinates. Raising an index with the metric, we set
$
\epsilon_i^{j}:=g^{jk}\epsilon_{ki}.
$
Since, in dimension two, for any local section
$A\in\Gamma(\Lambda^{1}_{1}(M)|_U)$ there exists a unique function
$\omega\in C^\infty(U)$ such that
\begin{equation}\label{eq:Aomega}
A_i^{j}=\omega\,\epsilon_i^{j},
\end{equation}
 the Kostant bundle admits a natural local identification
\begin{equation}\label{identification}
TM\oplus\Lambda^{1}_{1}(M)\;\cong\; TM\oplus\R ,\qquad
(Y,A)\longmapsto (Y,\omega).
\end{equation}
We now compute, in local coordinates, the Kostant connection induced on $TM\oplus\R$ via the vector bundle isomorphism \eqref{identification}, still denoted by $\widetilde\nabla$ with a slight abuse of notation.
Substituting  \eqref{eq:Aomega} into \eqref{kostantloc}, and using the fact that the Levi-Civita connection preserves both the metric and the volume form, we obtain
$$
\widetilde\nabla_{\partial_k}(Y,A)
=\Bigl(\bigl(Y^i_{,k}-\omega\,\epsilon_k^{i}\bigr)\partial_i\,,\,
\bigl(\omega_{,k}\,\epsilon_i^{j}+R^j_{i\ell k}Y^\ell\bigr)\partial_j\otimes dx^i\Bigr).
$$
Since,
on a $2$-dimensional pseudo-Riemannian manifold, 
$R^j_{i\ell k} =\frac{R}{2}\bigl(\delta^j_{\ell}\,g_{ik}-\delta^j_{k}\,g_{i\ell}\bigr)$,
the previous expression becomes
$$\widetilde\nabla_{\partial_k}(Y,A)
=\Bigl(\bigl(Y^i_{,k}-\omega\,\epsilon_k^{i}\bigr)\partial_i\,,\,
\bigl(\omega_{,k}\,\epsilon_i^{j}
+\tfrac{R}{2}\,(Y^j g_{ik}-\delta^j_k\,g_{i\ell}Y^\ell)\bigr)
\partial_j\otimes dx^i\Bigr).$$
Furthermore
$$
\bigl(\omega_{,k}\,\epsilon_i^{j}
+\tfrac{R}{2}\,(Y^j g_{ik}-\delta^j_k\,g_{i\ell}Y^\ell)\bigr)
\partial_j\otimes dx^i=\left(\omega_{,k}
+\tfrac{R}{2\sqrt{|\det g|}}\,(g_{1k}g_{2\ell}-g_{2k}g_{1\ell})Y^\ell\right)
\epsilon_i^{j}\,\partial_j\otimes dx^i.
$$
Therefore, under the identification \eqref{identification} (see also \eqref{eq:Aomega}) the Kostant connection induces a connection on $TM\oplus\R$ which, after  some algebraic manipulations, is given in local coordinates by
\begin{equation}\label{kostantloc2}
\widetilde\nabla_{\partial_k}(Y,\omega)
=\Bigl(\bigl(Y^i_{,k}-\omega\,\epsilon_k^{i}\bigr)\partial_i,\;
\omega_{,k}
-\tfrac{R}{2}\, g_{ki}\epsilon^{ij}g_{j\ell}Y^\ell\Bigr).
\end{equation}

%
%
%
\begin{Lemma}\label{lemma:canonical_line}
Let $(M,g)$ be a $2$-dimensional pseudo-Riemannian manifold. Suppose that the scalar curvature $R$ satisfies $\mathrm{d}R\neq 0$ and almost every point of $M$ admits  a neighbourhood where the space of Killing vector fields is nontrivial. Then the Kostant bundle admits a canonical\,\footnote{Depending only on the metric $g$.} rank-one subbundle with the property that any (local) $\widetilde\nabla$-parallel section necessarily takes values in this subbundle.
\end{Lemma}
\begin{proof}

Fix an open set $U\subset M$ sufficiently small to be covered by a single coordinate chart and such that a nontrivial Killing vector field $Y$ is defined at almost every its point.
As recalled above, $Y$ lifts to a $\widetilde\nabla$-parallel section $(Y,\nabla Y)$ of the Kostant bundle over $U$. Under the local bundle isomorphism
\eqref{identification} (see also \eqref{eq:Aomega}), this section corresponds to a
section $(Y,\omega)$ of $TM\oplus\R$ over $U$.
Since $(Y,\omega)$ is $\widetilde\nabla$-parallel, the local expression
\eqref{kostantloc2} yields
\begin{equation}\label{eq:parallel_system_Y}
Y^j_{,i}=\omega\,\epsilon_i^{\,j},
\qquad
\omega_{,i}=\frac{R}{2}\,g_{i a}\epsilon^{ab}g_{b\ell}Y^\ell.
\end{equation}
By covariantly differentiating the second equation in
\eqref{eq:parallel_system_Y} and by taking the skew-symmetric part, we obtain
$$
0=\bigl(R\,g_{i a}\epsilon^{ab}g_{b\ell}Y^\ell\bigr)_{,j}
 -\bigl(R\,g_{j a}\epsilon^{ab}g_{b\ell}Y^\ell\bigr)_{,i}
$$
as $\omega_{,ij}-\omega_{,ji}=0$. Moreover, by expanding this identity and by using $\nabla g=0$ and $\nabla\epsilon=0$, we get
\begin{equation*}
0=R_{,j}\,g_{i a}\epsilon^{ab}g_{b\ell}Y^\ell
 -R_{,i}\,g_{j a}\epsilon^{ab}g_{b\ell}Y^\ell
 +R\,g_{i a}\epsilon^{ab}g_{b\ell}Y^\ell_{,j}
 -R\,g_{j a}\epsilon^{ab}g_{b\ell}Y^\ell_{,i}.
\end{equation*}
Using the first equation in \eqref{eq:parallel_system_Y}, the last two terms cancel,
and we are left with
\begin{equation*}
R_{,j}\,g_{i a}\epsilon^{ab}g_{b\ell}Y^\ell
=
R_{,i}\,g_{j a}\epsilon^{ab}g_{b\ell}Y^\ell.
\end{equation*}
Since we are assuming by hypothesis $dR\neq 0$, it follows the existence of 
$\lambda\in C^\infty(U)$ such that
$g_{j a}\epsilon^{ab}g_{b\ell}Y^\ell=\lambda\,R_{,j}$.
For convenience, set $\widehat R_i:=\epsilon_{i a}g^{ab}R_{,b}$. It follows
immediately that
$
g_{i\ell}Y^\ell=\lambda\,\widehat R_i
$
and hence
\begin{equation}\label{eq:con1}
\widehat R_j\,g_{i\ell}Y^\ell=\widehat R_i\,g_{j\ell}Y^\ell.
\end{equation}
By covariantly differentiating \eqref{eq:con1} and by using again
\eqref{eq:parallel_system_Y}, we obtain the additional algebraic relation
\begin{equation}\label{eq:con2}
\widehat R_{j,k}\,g_{i\ell}Y^\ell+\omega\,\widehat R_j\,\epsilon_{ik}
=
\widehat R_{i,k}\,g_{j\ell}Y^\ell+\omega\,\widehat R_i\,\epsilon_{jk}.
\end{equation}
For each point $p\in U$, Equations \eqref{eq:con1} and \eqref{eq:con2} define a
homogeneous linear system whose coefficients are completely determined by the metric:
\begin{equation}\label{eq:system}
\begin{pmatrix}
\widehat R_{2}\,g_{11}-\widehat R_{1}\,g_{12}
&
\widehat R_{2}\,g_{12}-\widehat R_{1}\,g_{22}
&
0\\
\widehat R_{2,1}\,g_{11}-\widehat R_{1,1}\,g_{12} & \widehat R_{2,1}\,g_{12}-\widehat R_{1,1}\,g_{22}
& \widehat R_{1}\,\sqrt{|\det g|}\\
\widehat R_{2,2}\,g_{11}-\widehat R_{1,2}\,g_{12} & \widehat R_{2,2}\,g_{12}-\widehat R_{1,2}\,g_{22}
& \widehat R_{2}\,\sqrt{|\det g|}
\end{pmatrix}
\begin{pmatrix}
Y^{1}\\
Y^{2}\\
\omega
\end{pmatrix}
=
\begin{pmatrix}
0\\
0\\
0
\end{pmatrix}.
\end{equation}
Since $\mathrm dR\neq 0$, the covector $\widehat R$ is nowhere vanishing on $U$. The first row of the $3\times 3$ matrix in \eqref{eq:system} cannot vanish, as the non-degeneracy of $g$ would otherwise force $\widehat R=0$. Moreover, the third column of the matrix is non-zero, since it is proportional to $\widehat R$. It follows that the rank of the system \eqref{eq:system} is at least two at every point of $U$.

On the other hand, by assumption, a nontrivial Killing vector field is defined at almost every point of $U$. Such field lifts to a non-zero $\widetilde\nabla$-parallel section of the Kostant bundle and hence determines a nontrivial solution of \eqref{eq:system}. Therefore, the determinant of the matrix  in \eqref{eq:system} vanishes on a dense subset of $U$ and, by smoothness, 
vanishes identically on $U$. Then, the space of solutions of \eqref{eq:system} is one-dimensional  everywhere on $U$.
Since the coefficients of the linear system \eqref{eq:system} depend smoothly on $g$, its one-dimensional space of solutions varies smoothly with the base point. Hence, over $U$, these solution spaces define a smooth rank-one subbundle of the Kostant bundle.
Moreover, at each point $p\in M$, the corresponding line in the fiber is uniquely determined by the metric $g$, as the kernel of \eqref{eq:system} at $p$. Therefore, on overlaps of coordinate neighbourhoods, the locally defined line
subbundles necessarily agree. It follows that they glue together to define a global canonical rank-one subbundle of the Kostant bundle over $M$.

By construction, any $\widetilde\nabla$-parallel section of the Kostant bundle must take values in this subbundle. 
\end{proof}
%
%
%
%
%
%
\begin{Lemma}\label{prop.principal}
Let $(M,g)$ be a $2$-dimensional pseudo-Riemannian manifold. Let $p\in M$ such that $(dR)_p\neq 0$, where $R$ is  the scalar curvature of $g$. If for almost every point $q\in M\setminus \{p\}$ there exists a neighbourhood $U(q)\subset M$ of $q$ and a  nontrivial Killing vector field of $g|_{U(q)}$, then there exists a neighbourhood $U(p)$ of $p$ and a nowhere vanishing Killing vector field of $g|_{U(p)}$.
\end{Lemma}
\begin{proof}
Throughout the proof we implicitly restrict to a neighbourhood $U(p)\ni p$ such that
$\mathrm dR\neq 0$ on $U(p)$. All bundles and connections are therefore considered over
$U(p)$.
We consider the Kostant bundle $K=TM\oplus\Lambda^1_1(M)$ endowed with the 
connection $\widetilde\nabla$, together with the canonical rank-one subbundle
$L\subset K$ constructed in Lemma \ref{lemma:canonical_line}.

By assumption, in a neighbourhood of almost every point of $U(p)$ there exists a nontrivial
Killing vector field. Equivalently, at almost every point there exists a non-zero
$\widetilde\nabla$-parallel section of $K$.
Such a section spans $L$ on its domain, and since it is parallel, the covariant
derivative of any local section of $L$ along any vector field is again a section
of $L$. Therefore, the subbundle $L$ is totally geodesic with respect to
$\widetilde\nabla$ almost everywhere.
Since both $L$ and $\widetilde\nabla$ depend smoothly on the base point, this
property extends to every point of $U(p)$. Hence, $\widetilde\nabla$ induces a connection
on the rank-one bundle $L$.

Every Killing vector field on $U(p)$ corresponds to a $\widetilde\nabla$-parallel section
of $K$, and hence to a parallel section of the induced connection on $L$.
The existence of a nontrivial parallel section of a connection on a line bundle implies
that the curvature of the connection vanishes. Since by assumption such a section exists
in a neighbourhood of almost every point of $U(p)$, the curvature of the induced connection
on $L$ vanishes almost everywhere on $U(p)$, and therefore vanishes identically by
smoothness.

Since the induced connection on $L$ is flat, there exists a nontrivial parallel
section in a neighbourhood of every point of $U(p)$, and in particular in a neighbourhood
of $p$. This section corresponds to a local Killing vector field.

Finally, using the local description \eqref{eq:system} of the canonical line $L$,
we note that if $Y^1=Y^2=0$ at a point $q$, then necessarily $\omega=0$ at $q$ as well,
because the third column of the system is proportional to $\widehat R$ and $\mathrm dR\neq 0$.
Hence, a parallel section $(Y,\omega)$ that vanishes in its $TM$-component at $q$ must be
identically zero. It follows that the Killing vector field obtained above is nowhere
vanishing in a neighbourhood of $p$, which concludes the proof.
\end{proof}

\begin{Lemma}\label{prop.killing.morse}
Let $(M,g)$ be a $2$-dimensional pseudo-Riemannian manifold. Let $p\in M$ be a Morse singularity of the scalar curvature $R$ of $g$. If for almost every point $q\in M$ there exists a neighbourhood $U(q)\subset M$ of $q$ and a nontrivial Killing vector field of $g|_{U(q)}$, then there exists a neighbourhood $U(p)$ of $p$ and a Killing vector field of $g|_{U(p)}$ that vanishes at the point $p$ only.
\end{Lemma}
\begin{proof}
Let $p\in M$ be a Morse critical point of the scalar curvature $R$ of $g$.
Shrinking a neighbourhood $U(p)$ of $p$ if necessary, we may assume that $p$ is the only
critical point of $R$ in $U(p)$; in particular $\mathrm dR\neq 0$ on $U(p)\setminus\{p\}$.

Fix local coordinates $(x,y)$ on $U(p)$ such that $p=(0,1)$, and assume that the rectangle
$$
\{(x,y)\mid -1\le x\le 1,\,-1\le y\le 2\}
$$
is contained in the coordinate chart. Let $q=(0,0)$.
Since $\mathrm dR\neq 0$ at $q$, Lemma \ref{prop.principal} yields a nowhere vanishing
local Killing vector field near $q$, equivalently a non-zero $\widetilde\nabla$-parallel
section of the Kostant bundle $K$. Using the identification $K\simeq TM\oplus\R$, we write
this section as
$$
\left(
Y^1(x,y)\,,\,
Y^2(x,y)\,,\,
\omega(x,y)
\right)
$$
defined (after shrinking $\varepsilon>0$) in particular  for all $(x,y)$ with $x\in(-\varepsilon,\varepsilon)$
and $y=0$.

By Lemma \ref{lemma:canonical_line}, on $U(p)\setminus\{p\}$ the Kostant bundle admits a
canonical rank-one subbundle $L\subset K$, and every $\widetilde\nabla$-parallel section
takes values in $L$; hence such sections are unique up to multiplication by a constant.

We now consider the 1-parameter family of curves
$$
\sigma_s:[0,2]\to U, \     \sigma_s(t)= (s, t), \ s \in (-\varepsilon, \varepsilon)
$$
and we parallel-transport the vector
$
\left(
Y^1(s,0)\,,\,
Y^2(s,0)\,,\,
\omega(s,0)
\right)
$
along $\sigma_s$ with respect to $\widetilde\nabla$. This yields a smooth section $S$ of $K$
defined for $x\in(-\varepsilon,\varepsilon)$ and $y\in(0,2)$, satisfying
$\widetilde\nabla_{\partial_y}S=0$.

For $s\neq 0$, the curve $\sigma_s$ is contained in $U(p)\setminus\{p\}$; by Lemma \ref{prop.principal}, there exists locally along $\sigma_s$ a non-zero $\widetilde\nabla$-parallel
section, which must coincide with $S$ up to a constant since both take values in $L$ and
agree at $y=0$. Thus $S$ is $\widetilde\nabla$-parallel on $\{x\neq 0\}$.
By continuity, $S$ is $\widetilde\nabla$-parallel also along $\{x=0\}$, hence on a
neighbourhood of $p$.

Therefore $S$ corresponds to a local Killing vector field $Y$ defined near $p$.
Since the flow of a Killing vector field preserves $R$, it preserves the set of critical
points of $R$. As $p$ is an isolated critical point in $U(p)$, the flow fixes $p$, and thus
$Y(p)=0$.

Finally, $Y$ has no other zeros near $p$: indeed, on $U(p)\setminus\{p\}$ we have
$\mathrm dR\neq 0$, and Lemma \ref{prop.principal} implies that any nontrivial local
Killing vector field is nowhere vanishing there. Hence $Y$ vanishes only at $p$.
Lemma is proved.
\end{proof}

\section*{Acknowledgements}
G. M. was supported by the project ``Finanziamento alla Ricerca'' under the contract numbers 53\_RBA21MANGIO, and by the PRIN project 2022 ``Real and Complex Manifolds: Geometry and Holomorphic Dynamics'' (code 2022AP8HZ9). 
V. M. was supported by the DFG project 529233771 and the ARC Discovery Programme DP210100951.
F. S. was supported by the ``Starting Grant'' under the contract number 53\_RSG22SALFIL. G. M. and F. S. are members of the GNSAGA of the INdAM.

\subsection*{Declarations}

\subsubsection*{Associated data}

The authors declare that no 
no data was used for the research described in the article.

\subsubsection*{Conflict of interest }
The authors declare that they have no conflict of interest.

\end{document}